\documentclass[a4paper,oneside,10pt]{amsart}

\usepackage{a4wide}
\usepackage[T1]{fontenc}
\usepackage[ansinew]{inputenc}
\usepackage{lmodern} 
\usepackage{graphicx}
\usepackage{amsmath}
\usepackage{amsthm}
\usepackage{amsfonts}
\usepackage{amssymb}
\usepackage{setspace}
\usepackage{mathrsfs}
\usepackage[all]{xy}
\usepackage{enumerate}
\usepackage{xcolor}

\theoremstyle{plain}

\newtheorem{theorem}{Theorem}[section]

\newtheorem{corollary}[theorem]{Corollary}
\newtheorem{proposition}[theorem]{Proposition}
 \newtheorem{lemma}[theorem]{Lemma}

\theoremstyle{definition}
\newtheorem{remark}[theorem]{Remark}

\newtheorem{assumption}[theorem]{Assumption}

 \newtheorem{definition}[theorem]{Definition}

\newtheorem*{theorem*}{Theorem}
\newtheorem*{proposition*}{Proposition}
\newtheorem*{definition*}{Definition}
\numberwithin{equation}{section}

\def\chookrightarrow{\hookrightarrow}
\def\dprod#1#2{\langle #1,#2\rangle}
\def\Dprod#1#2{\Big\langle #1,#2\Bigr\rangle}

\def\semis{\mathcal{P}}
\def\xX{\underline{X}}

\def\TT{\overline{T}}

\def\dom{D}
\def\inj{\hookrightarrow}
\def\dd{\mathrm{d}}
\def\ee{\mathrm{e}}
\def\Ell{\mathrm{L}}
\def\tlim{\mathop{\tau\mathrm{lim}}}
\def\veps{\varepsilon}
\def\LLL{\mathscr{L}}
\def\Fav{F}
\def\rest{|}

\def\XHol_#1{\underline{X}_{#1}}
\def\A{\mathcal{A}}
\def\D{\mathscr{E}}
\def\C{\mathrm{C}}
\def\W{\mathrm{W}}
\def\h{\mathrm{h}}
\def\Distr{\mathscr{D}'}
\def\downto{\downarrow}
\def\Id{I}

\def\TFav_#1{F_{#1}}
\def\THol_#1{\underline{X}_{#1}}

\def\conti{\mathrm{cont}}

\newenvironment{abc}{\begin{enumerate}[{\rm (a)}]}{\end{enumerate}}
\newenvironment{num}{\begin{enumerate}[{\rm 1.}]}{\end{enumerate}}
\newenvironment{iiv}{\begin{enumerate}[{\rm (i)}]}{\end{enumerate}}

\begin{document}
\title{Intermediate and Extrapolated Spaces for Bi-Continuous Semigroups} 
\author{Christian Budde}
\address{University of Wuppertal, School of Mathematics and Natural Sciences, Gaussstrasse 20, 42119 Wuppertal, Germany}
\email{cbudde@uni-wuppertal.de}
\author{B\'alint Farkas}
\address{University of Wuppertal, School of Mathematics and Natural Sciences, Gaussstrasse 20, 42119 Wuppertal, Germany}
\email{farkas@math.uni-wuppertal.de}

\begin{abstract}
We discuss the construction of the full Sobolev (H\"older) scale for non-densely defined operators on a Banach space with rays of minimal growth. In particular, we give a construction for extrapolation- and Favard spaces of generators of (bi-continuous) semigroups, or which is essentially the same, Hille--Yosida operators on \emph{Saks spaces}. 
\end{abstract}

\date{}
\maketitle

\def\BC{\mathrm{C}_{\mathrm{b}}}
\def\BUC{\mathrm{UC}_{\mathrm{b}}}
\def\Lipb{\mathrm{Lip}_{\mathrm{b}}}
\def\dom{D}
\def\txX{\widetilde{X}}
\def\xX{\underline{X}}
\def\NN{\mathbb{N}}
\def\CC{\mathbb{C}}

\def\ZZ{\mathbb{Z}}

\def\RR{\mathbb{R}}
\def\aA{\underline{A}}
\def\tT{\underline{T}}
\def\calP{\mathcal{P}}

\section*{Introduction}
Extrapolation spaces for generators of \emph{$C_0$-semigroups} (used here synonymously to  \emph{`strongly continuous, one-parameter operator semigroups of bounded linear operators'}) on Banach spaces, or for more general operators, have been designed to study maximal regularity questions by Da Prato and Grisvard \cite{DAPRATO1984107}; see also Walter \cite{Walter}, Amann \cite{Amann}, van Neerven \cite{van1992adjoint},  Nagel, Sinestrari \cite{nagel1993inhomogeneous}, Nagel \cite{NagelSurvey}, and Sinestrari \cite{Sinestrari1996}. These spaces (and the extrapolated operators) also play a central role in certain abstract perturbation results, most prominently for boundary-type perturbations, see e.g., Desch, Schappacher \cite{Desch},  Greiner \cite{Greiner}, Staffans, Weiss \cite{Staffans2004},  Adler, Bombieri, Engel \cite{adler2014}, and as a result their application area is vast. 

\medskip
In this paper, we concentrate on the construction of extrapolation spaces for linear operators $A$ having non-empty resolvent set, but we do not assume the operator to be a Hille--Yosida operator or to be densely defined. For the densely defined case such a construction is known due to the seminal papers of  Da Prato, Grisvard, \cite{DaPrato1982}, Amann \cite{Amann} and Nagel, Sinestrari \cite{nagel1993inhomogeneous}. In the case of non-densely defined, sectorial operators there is a very general---\emph{almost purely algebraic}---construction due to Haase \cite{Haase}, who also discusses  universal extrapolation spaces. Here, we present a slightly different construction of extrapolation- and extrapolated Favard spaces, leading to  the construction of \emph{extrapolated semigroups} in the absence of norm strong continuity.  For a non-densely defined Hille--Yosida operator $A$ on the Banach space $X_0$ such a construction is possible by taking the part of $A$ in $\xX_0:=\overline{\dom(A)}$, on which space the restricted operator becomes a generator of a $C_0$-semigroup, so one can construct an extrapolated semigroup on the extrapolation space $\xX_{-1}$, see Nagel, Sinestrari \cite{NS}. But this semigroup will usually not leave the original Banach space $X_0$ invariant. This is why we restrict our attention to the situation where strong continuity of the semigroup is guaranteed with respect to some coarser locally convex topology. Here the framework of bi-continuous semigroups, or that of Saks spaces, (see K\"uhnemund \cite{Ku} and Section \ref{sec:bicont}) appears to be adequate. However, most of the results presented here are valid also for generators of other classes of semigroups: integrable semigroups of Kunze \cite{Kunze2009}, ``C-class'' semigroups of Kraaij \cite{Kraaij2016}, $\pi$-semigroups of Priola \cite{Priola}, weakly continuous semigroups of Cerrai \cite{Cerrai} to mention a few.

\medskip Given a Banach space $X_0$, a Hausdorff locally convex topology $\tau$ on $X_0$ (with certain properties described in Section \ref{sec:bicont}), a bi-continuous semigroup $(T(t))_{t\geq 0}$ with generator $A$, we construct the full scale of abstract Sobolev (or H\"older) and Favard spaces $X_\alpha$, $\xX_\alpha$, $F_\alpha$ for $\alpha\in \RR$, and the corresponding extrapolated semigroups $(T_\alpha(t))_{t\geq 0}$. (If $\tau$ is the norm topology, there is nothing new here, and everything can be found in \cite[Section II.5]{EN}.) These constructions, along with some  applications, form the main content of this paper. Here we illustrate the results on the following well-known example (see also Nagel, Nickel, Romanelli \cite{NagelIdent} and Section \ref{sec:examp} for details): Consider the Banach space $X_0:=\BC(\RR)$ of bounded, continuous functions and the (left) translation semigroup $(S(t))_{t\geq 0}$ thereon, defined by $(S(t)f)(x)=f(x+t)$, $x\in\RR$, $t\geq 0$, $f\in X_0$. For $\alpha\in (0,1)$
\[
\BC^1(\RR)\hookrightarrow \Lipb(\RR) \hookrightarrow\h_b^{\alpha}(\RR)\chookrightarrow \h_{b,\text{loc}}^{\alpha}(\RR)\chookrightarrow\BC^{\alpha}(\RR)\hookrightarrow\BUC(\RR)\chookrightarrow \BC(\RR)\chookrightarrow \Ell^{\infty}(\RR),
\]
where $\BC^1$ is the space of differentiable functions with derivative in $\BC$, $\Lipb(\RR)$ is the space of bounded, Lipschitz functions, $\h_b^{\alpha}$ is the space of bounded, little-H\"older continuous functions, $\h_{b,\text{loc}}^{\alpha}$ is the space of bounded, locally little-H\"older continuous functions, $\BC^\alpha$ is the space of bounded, H\"older continuous functions,  $\BUC(\RR)$ is the space of bounded, uniformly continuous functions. From the abstract perspective and using our notation this corresponds to 
the inclusions of Banach spaces:
\[
X_1\hookrightarrow F_1\hookrightarrow \xX_{\alpha}\chookrightarrow X_{\alpha}\chookrightarrow F_{\alpha}\hookrightarrow\xX_0\chookrightarrow X_0\chookrightarrow  F_0.
\]
 The extension of the previous diagram for the full scale $\alpha\in \RR$ is possible by extrapolation. The spaces $\xX_\alpha$ and $F_\alpha$ ($\alpha\in (0,1)$) are well studied and we refer to  the books by Lunardi \cite{Lunardi} and by Engel, Nagel \cite[Section II.5]{EN} for a systematic treatment. However, the definition of $X_\alpha$ is new, and requires a recollection of results concerning the other two kinds of spaces.

\medskip Extrapolated Favard spaces are not only important from the perturbation theoretic point of view: They sometimes help to reduce problems concerning semigroups being no strongly continuous, to the study of the underlying $C_0$-semigroup. This perspective is propagated by Nagel and Sinestrari in \cite{NS}: To any Hille--Yosida operator on $X_0$ one can construct a Banach space $F_0$ (the Favard class) containing $X_0$ as a closed subspace,  and a semigroup $(T(t))_{t\geq 0}$  on $F_0$.
We adapt this point of view also in this paper. In particular, we provide an alternative (and short) proof of the Hille--Yosida type generation theorem for bi-continuous semigroups (due to K\"uhnemund \cite{Ku}) by employing solely the $C_0$-theory. Note, however, that the semigroup $(T(t))_{t\geq 0}$ defined on $F_0$ may not leave $X_0$ invariant in general, this is the issue, where the additional topology $\tau$ can be helpful. 

\medskip Applications of the Sobolev (H\"older) scale, as presented here, to perturbation theory, in the spirit of the results of Desch, Schappacher \cite{Desch}, or of Jacob, Wegner, Wintermayr \cite{JWW}, will be presented in a forthcoming paper.

\medskip This work is organized as follows: In Section \ref{sec:invert} we recall the standard constructions and result for extrapolation spaces for densely defined (invertible) operators. Moreover, we construct extrapolation spaces for not densely defined operators $A$ with $\dom(A^2)$ dense in $\dom(A)$ for the norm of $X_0$. Our argumentation differs form the one in Haase \cite{Haase} in that we build the space $X_{-1}$ based on $\xX_{-2}$ (which, in turn, arises from $\xX_0$ and $\xX_{-1}$), i.e., in a bottom-to-top and back to-bottom manner, resulting in the continuous inclusions
\[
\xX_0\hookrightarrow X_0\hookrightarrow \xX_{-1}\hookrightarrow X_{-1}\hookrightarrow\xX_{-2}.
\] 
(All these inclusions are not surjective in general.)
This approach becomes convenient when we compare the arising extrapolation spaces $X_{-1}$ and $\xX_{-1}$ and construct the extrapolated semigroups. In Section \ref{sec:minimal} we turn to intermediate spaces; the results there are classical, but are put in the general perspective of this paper. We also present a method for a `concrete' representation of extrapolation spaces. Section \ref{sec:sgrps} discusses the Sobolev (H\"older) scale for semigroup generators, and has a survey character. In Section \ref{sec:bicont} we recall the concept of bi-continuous semigroups, construct the corresponding extrapolated semigroups and give a direct proof of the Hille--Yosida generation theorem (due to K\"uhnemund \cite{Ku}) which uses extrapolation techniques.  We conclude this paper with some examples in Section \ref{sec:examp}, where we determine the extrapolation spaces of concrete semigroup generators. Among others the previously mentioned example of the translation semigroup (complementing results of Nagel, Nickel, Romanelli \cite[Sec.3.1, 3.2]{NagelIdent}) and then \emph{implemented semigroups} (cf.{} Alber \cite{Alber2001}) are discussed in  detail.

\section{Spaces for invertible operators}\label{sec:invert}
In this section we construct abstract Sobolev (H\"older) and extrapolation spaces (the so-called Sobolev scale) for a boundedly invertible linear operator $A$ defined on a Banach space. Some of the results are well-known and nowadays even standard, but we chose to include them here for a sake of completeness, and also because a review of these is necessary for the construction of spaces when we deal with not densely defined operators. The emphasis will be, however, put on this case, when the construction of these extrapolation spaces is new, see Section \ref{sec:extra} below. We also note that everything in what follows is also valid for operators on Fr\'echet spaces, an except for some assertions one does not even need metrizability.

\medskip\noindent 
Let $X_0$ be Banach space, and let $A:\dom(A)\to X_0$ be a not necessarily densely defined linear operator with non-empty resolvent set $\rho(A)\neq\emptyset$. As a matter of fact, for convenience we suppose $0\in \rho(A)$. If this was not so, then by taking $\lambda\in \rho(A)$ we may consider $A-\lambda$ instead of $A$ and carry out the constructions for this new operator, for which in fact $0\in \rho(A-\lambda)$. The arising spaces will not depend on $\lambda\in\rho(A)$ (up to isomorphism).
\subsection{Abstract Sobolev spaces} 
The material presented here is known, see Nagel \cite{NagelSobolev}, Nagel, Nickel, Romanelli \cite{NagelIdent} or Engel, Nagel \cite[Section II.5]{EN}, and some parts are valid even for operators on locally convex spaces.
We set $X_1:=\dom(A)$ which becomes a Banach space if endowed with the graph norm 
\[
\|x\|_{A}:=\|x\|+\|Ax\|.
\]
An equivalent norm is given by $\|x\|_{X_1}:=\|Ax\|$, since we have assumed $0\in\rho(A)$.
Then we have the isometric isomorphism
\[
A:X_1\to X_0\quad \text{with inverse}\quad A^{-1}:X_0\to X_1.
\]
\begin{definition} Suppose $0\in\rho(A)$ here and throughout in the following. Let $n\in \NN$.
\begin{abc}
\item We define
\[
X_n:=\dom(A^n)\quad\text{and}\quad \|x\|_{X_n}:=\|A^n x\|\quad\text{for $x\in X_n$}.
\]
If we want to stress the dependence on $A$, then we write $X_{n}(A)$ and $\|\cdot\|_{X_{n}(A)}$.
\item Let
\[
X_{\infty}(A):=\bigcap_{n\in\NN} X_n,
\] 
often abbreviated as $X_\infty$.
\item We further set 
\[
\xX_0:=\overline{\dom(A)}, \quad \aA:=A\rest_{\xX_0},
\]
the part of $A$ in $\xX_0$, i.e.,
\[
\dom(\aA)=\bigl\{x\in\dom(A):Ax\in \xX_0\bigr\}.
\]
 Moreover, we let
 \[
 \xX_n:=\dom (\aA^n), \quad \|x\|_{\xX_n}:=\|\aA^n x\|.
 \]
To be specific about the underlying operator $A$ we write $\xX_n(A)$ and $\|x\|_{\xX_n(A)}$.

\item For $n\in\NN$ we set $A_{n}:=A\rest_{X_n}$, the part of $A$ in $X_n$, in particular $A_0=A$. Similarly, we let $\aA_{n}:=\aA|_{\xX_n}$, for example $\aA_0=\aA$. By this notation we also understand implicitly that the surrounding space is $X_n(A)$ respectively $\xX_n(A)$ with its norm, see the next Remark \ref{rem:1}.
\end{abc}
\end{definition}

\begin{remark}\label{rem:1}
\begin{num}
\item The choice of the notation $X_0$ and $\xX_0$ and the like should be self-explanatory: By ``underlining'' we always indicate an object which is in some sense smaller than the one without underlining. The space $\xX_0(A)$ is connected with the domain of $\dom(A)$, and the whole issue of distinguishing between $X_0$ and $\xX_0$ becomes interesting only if $A$ is not densely defined but its part $\aA$ \emph{is} (cf. Remark \ref{rem:2}). We will stick to the notation $\aA$ for the part of the operator $A$ instead of $A\rest_{\xX_0}$, because it fits better with our general notation.
\item If $A$ is densely defined, then $X_n(A)=\xX_n(A)$ for each $n\in \NN$. In particular, if $\xX_1(A)=\dom(\aA)$ is dense in $\xX_0(A)$, then $\xX_n(A)=\xX_n(\aA)$ for each $n\in\NN$.
\item For $n\in\NN$ we evidently have $X_1(A^n)=X_n(A)$. Also $\xX_1(A^n)=\xX_n(A)$ holds, because $\dom(\aA^n)=\dom(\underline{A^n})$. Indeed, the inclusion ``$\dom(\aA^n)\subseteq \dom(\underline{A^n})$'' is trivial. While for $x\in \dom(\underline{A^n})$ we have $x\in \xX_0$ and $A^nx\in \xX_0$, implying $A^{n-1}x\in \dom(\aA)$, and then recursively $x\in \dom(\aA^n)$.
\item For $x\in \dom(A_n)=\dom(A^{n+1})$ we have $\|x\|_{X_1(A_n)}=\|A_nx\|_{X_n(A)}=\|A^{n+1}x\|=\|x\|_{X_{n+1}(A)}$. Similarly $\dom(\aA_n)=\dom(\aA^{n+1})$.
\end{num}
\end{remark}

\begin{proposition}Suppose $\aA$ is densely defined in $\xX_0$.
\begin{abc}
\item For $n\in \NN$ the mappings $A^n:X_n\to X_0$ and $\aA^n:\xX_n\to \xX_0$ are isometric isomorphisms. 
\item For $n\in\NN$ the operators $A_n:X_{n+1}\to X_{n}$ and $\aA_n:\xX_{n+1}\to \xX_{n}$ are isometric isomorphisms that intertwine $A_{n+1}$ and $A_n$, respectively, $\aA_{n+1}$ and $\aA_n$.
\item If $\dom(\aA)$ is dense in $\xX_0$, then $X_\infty$ is dense in $\xX_n$ for each $n\in\NN$. As a consequence $\xX_m$ is dense in $\xX_n$ for each $m,n\in \NN$ with $m\geq n$.
\end{abc}
\end{proposition}
\begin{proof} The statements (a) and (b) are trivial by construction. 

\medskip\noindent (c) This is \cite[Thm. 6.2]{Arendt94} due to Arendt, El-Mennaoui and K\'eyantuo, because $\aA$ is densely defined in $\xX_0$.
\end{proof}

\begin{remark}
We note that the proof of the previous assertion (c) in \cite[Thm. 6.2]{Arendt94} is based on a Mittag-Leffler type result due to Esterle \cite{Esterle} which is valid in complete metric spaces. Hence the previous statements (a), (b) and (c) are all remain true for Fr\'echet-spaces with verbatim the same proof as in \cite{Arendt94}.
\end{remark}
Henceforth, another standing assumption in this paper (though not everywhere needed) is that $\aA$ is in $\xX_0$ densely defined, i.e.,
\[
\overline{\dom(\aA)}=\xX_0.
\]

\begin{remark}\label{rem:2} The condition of $\dom(\aA)$ being dense in $\xX_0$ can be for example assured if there are $M,\omega>0$ such that $(\omega,\infty)\subseteq \rho(A)$ and \begin{equation}\label{eq:wHY1} \|\lambda R(\lambda,A)\|\leq M\quad\text{for all $\lambda>\omega$}. \end{equation} Indeed,
in this case we have for $x\in \dom(A)$ \[ \|\lambda R(\lambda, A)x-x\|=\|R(\lambda,A)Ax\|\leq \frac{M\|Ax\|}{\lambda}\to 0\quad\text{for $\lambda\to \infty$}. \] Hence $\dom(A^2)\subseteq \dom(\aA)$ is dense in $\dom(A)$ for the norm of $X_0$, and this implies the density of $\dom(\aA)$ in $\xX_0$. An
operator $A$ satisfying \eqref{eq:wHY1} is often said to have a \emph{ray of minimal growth}, see, e.g., \cite[Chapter 3]{Lunardi}, and also Section \ref{sec:minimal} below. Another term being used is \emph{``weak Hille--Yosida operator''}. \end{remark}

\begin{proposition}
If $T\in\LLL(X_0)$ is a linear operator commuting with $A^{-1}$, then the spaces $X_n$ and $\xX_n$ are $T$-invariant, and $T\in \LLL(X_n)$ for $n\in \NN$.
\end{proposition}
\begin{proof}The condition means that $Tx\in \dom(A)$ for each $x\in\dom(A)$ and for such $x$ we have $ATx=TAx$. This implies the invariance of $X_1$ and that $\|Tx\|_{X_1(A)}\leq \|T\|\|x\|_{X_{1}(A)}$. Using the boundedness assumption we see that $\xX_1$ too stays invariant under $T$. For general $n\in \NN$ we may argue by recursion, or simply invoke Remark \ref{rem:1}.
\end{proof}

\subsection{Extrapolation spaces}\label{sec:extra}
The construction for the extrapolation spaces here is known and more or less standard if $A$ is densely defined, or if $A$ is Hille--Yosida operator, see, e.g., \cite{NS}. 

For $x\in X_0$ we define $\|x\|_{\xX_{-1}(A)}:=\|A^{-1}x\|$. Then the surjective mapping
\[
A:(\dom(A),\|\cdot\|)\to (X_0,\|\cdot\|_{\xX_{-1}(A)})
\]
becomes isometric, and hence has a uniquely determined continuous extension
\[
\aA_{-1}:(\xX_0,\|\cdot\|)\to(\xX_{-1},\|\cdot\|_{\xX_{-1}(\aA)}), 
\]
which is then an isometric isomorphism, where $(\xX_{-1},\|\cdot\|_{\xX_{-1}(\aA)})$
denotes a completion of \\$(\xX_0,\|\cdot\|_{\xX_{-1}(A)})$. Of course, uniqueness is given as soon as the completion is fixed.
By construction we obtain immediately:

\begin{proposition}\label{prop:altercompl}
$X_0$ is continuously and densely embedded in $\xX_{-1}$. If $\aA$ is densely defined in $\xX_0$, then also $X_\infty$ is dense in $\xX_{-1}$. As a consequence
$(\xX_{-1},\|\cdot\|_{\xX_{-1}(\aA)})$ is the completion of $(\xX_0,\|\aA^{-1}\cdot \|)$. 
\end{proposition}

\begin{proof}
Of course $X_0$is dense in $\xX_{-1}$ because of completion. For $x\in X_0$ we have
\[
\|x\|_{\xX_{-1}(\aA)}=\|AA^{-1}x\|_{\xX_{-1}(\aA)}=\|\aA_{-1}A^{-1}x\|_{\xX_{-1}(A)}\leq \|\aA_{-1}\|\cdot \|A^{-1}x\|\leq \|\aA_{-1}\|\cdot \|A^{-1}\|\cdot \|x\|,
\]
showing the continuity of the embedding. The last assertion follows since $X_\infty$ is dense in $\dom(A)$ with respect to $\|\cdot\|$. 
\end{proof}

Of course one can iterate the whole procedure an obtain the following chain of dense and continuous embeddings
\[
\xX_{0}\inj \xX_{-1}\inj \xX_{-2}\inj \cdots\inj \xX_{-n} \quad\text{for $n\in \NN$},
\]
where for $n\geq 1$ the space $\xX_{-n}$ is a completion of $\xX_{-n+1}$ with respect to the norm $\|x\|_{\xX_{-n}(\aA)}$ defined by $\|x\|_{\xX_{-n}(\aA)}=\|\aA^{-1}_{-n+1}x\|_{\xX_{-n+1}(\aA)}$ and 
\[
\aA_{-n}:\xX_{-n+1}\to \xX_{-n} 
\]
is a unique continuous extension of $\aA_{-n+1}:\dom(\aA_{-n+1})\to \xX_{-n+1}$ to $\xX_{-n}$. 

\medskip\noindent These spaces, just as well the ones in the next definition, are called \emph{extrapolation spaces} for the operator $A$, see, e.g., \cite{NS} or \cite[Section II.5]{EN} for the case of semigroup generators. The spaces $\xX_{-1}$ and $\xX_{-2}$, just as well the operator $\aA_{-2}$ will be used to define the extrapolation space $X_{-1}(A)$. To this purpose we identify $X_0$ with a subspace of $\xX_{-1}$ and of $\xX_{-2}$.
\begin{definition}\label{def:extra}
Consider $X_0$ as a subspace of $\xX_{-2}$, and define
\[
X_{-1}:=\aA_{-2}(X_0):=\bigl\{\aA_{-2}x:\ x\in X_0\bigr\}\quad \text{and}\quad \|x\|_{X_{-1}}:=\|\aA_{-2}^{-1}x\|.
\]
Furthermore we set $A_{-1}:=\aA_{-2}|_{X_0}$, the part of $\aA_{-2}$ in $X_0$.
Again, $X_{-1}(A)$ and $\|\cdot\|_{X_{-1}(A)}$ make the notation unambiguous.
\end{definition}

 In what follows, we will define higher order extrapolation spaces and prove that all these spaces line up in a scale, where one can switch between the levels with the help of (a version) of the operator $A$ (or $A_{-1}$).
 
\begin{proposition}
The operator $A_{-1}$ is an extension of $\aA_{-1}$, 
$(X_{-1},\|\cdot\|_{X_{-1}})$ is a Banach space, the norms of $\xX_{-1}$ and $X_{-1}$ coincide on $\xX_{-1}$, and $\xX_{-1}$ is a closed subspace of $X_{-1}$. The mapping $A_{-1}:X_0\to X_{-1}$ is an isometric isomorphism.
\end{proposition}
\begin{proof}The first assertion is true because $\aA_{-2}$ is an extension of $\aA_{-1}$.
That $X_{-1}$ is a Banach space is immediate from the definition. Since $\aA_{-2}^{-1}\aA_{-1}=\Id$ on $\xX_0$, for $x\in \xX_{-1}$ we have $\aA_{-1}^{-1}x\in \xX_0\subseteq X_0$, so that $\|\aA_{-2}^{-1}x\|=\|\aA_{-2}^{-1}\aA_{-1}\aA^{-1}_{-1}x\|=\|\aA^{-1}_{-1}x\|=\|x\|_{\xX_{-1}}$. This establishes that the norms coincide. Since $\xX_{-1}$ is a Banach space (with its own norm), it is a closed subspace of $X_{-1}$. That $A_{-1}$ is an isometric isomorphism follows from the definition.
\end{proof}

\begin{remark}\label{rem:iter-1}
 For any $n\in\NN$ we have by construction that $\xX_{-1}(\aA_{-n})=\xX_{-(n+1)}(\aA)$ and $X_{-1}(A_{-n})=X_{-(n+1)}(A)$.
\end{remark}

\begin{proposition}\label{prop:intertwine+1}
For $n\in\ZZ$ the operators $A_n:X_{n+1}\to X_{n}$ and $\aA_n:\xX_{n+1}\to \xX_{n}$ are isometric isomorphisms that intertwine $A_{n+1}$ and $A_n$, respectively, $\aA_{n+1}$ and $\aA_n$.
\end{proposition}
\begin{proof} For $n\in \NN$ these have been proved in Proposition \ref{prop:intertwine+1}.
So we assume $n<0$. For $n=-1$ the statement about isometric isomorphisms is just the definition, and the intertwining property is also evident. By recursion we obtain the validity of the assertion for general $n\leq -1$ and for the operator $\aA_n$. By Remark \ref{rem:iter-1} it suffices to prove that $A_{-1}$ intertwines $A_{-1}$ and $A_0$. For $x\in \dom(A_0)=\dom(A)$ we have $A_{-1}x\in X_0=\dom(A_{-1})$ and $Ax=A^{-1}_{-1}A_{-1}A_{-1}x$.
\end{proof}

Thus for $n\in \NN$ we have the following chain of embeddings (continuous, dense, denoted by $\inj$) and inclusions as closed subspaces (denoted by $\subseteq$):
\[
\cdots\inj \xX_{n}\subseteq X_n\inj \xX_0\subseteq X_0\inj \xX_{-1}\subseteq X_{-1}\inj \xX_{-2}\subseteq X_{-2}\inj \cdots \xX_{-n}\subseteq X_{-n}\inj\cdots,\]
where in general the inclusions are strict (see the examples in Section \ref{sec:examp}). We also have the following chain of isometric isomorphisms
\[
\cdots\longrightarrow\xX_{n+1}\stackrel{\aA_n^{-1}}{\longrightarrow} \xX_{n}\longrightarrow\cdots \longrightarrow\xX_{1}\stackrel{\aA_{0}^{-1}}{\longrightarrow} \xX_{0}\stackrel{\aA_{-1}^{-1}}{\longrightarrow}\xX_{-1}{\longrightarrow}\cdots \longrightarrow \xX_{-n+1}\stackrel{\aA_{-n}^{-1}}{\longrightarrow}\xX_{-n}\longrightarrow \cdots
\]
and
\[
\cdots\longrightarrow X_{n+1}\stackrel{A_n^{-1}}{\longrightarrow} X_{n}\longrightarrow\cdots \longrightarrow X_{1}\stackrel{A_{0}^{-1}}{\longrightarrow} X_{0}\stackrel{A_{-1}^{-1}}{\longrightarrow}X_{-1}{\longrightarrow}\cdots \longrightarrow X_{-n+1}\stackrel{A_{-n}^{-1}}{\longrightarrow}X_{-n}\longrightarrow \cdots.
\]

\begin{proposition}
\begin{abc}
\item $\xX_1({\aA_{-1}})=\xX_0$ and $X_1(A_{-1})=X_0$ with the same norms.
\item $\xX_{-1}({\aA_{1}})=\xX_0$ with the same norms.
\item $(\aA_1)_{-1}=\aA$.
\item $X_{-1}(A_{1})=X_0$ with the same norms, and $(A_1)_{-1}=A$.
\end{abc}
\end{proposition}
\begin{proof} 
(a) By definition $X_1(A_{-1})=\dom(A_{-1})=X_0$ with the graph norm of $A_{-1}$. Since $A_{-1}$ extends $A$, for $x\in X_0$ we have $\|A_{-1}x\|_{X_{-1}(A)}=\|Ax\|_{-1}=\|A^{-1}Ax\|=\|x\|$. The first statement then follows as well, because $\xX_1({\aA_{-1}})=X_1(\aA_{-1})=\overline{\dom(\aA)}=\xX_0$ with the same norms.

\medskip\noindent (b) For $x\in \xX_1(\aA)=\dom(\aA^2)$ we have 
\[
\|x\|_{\xX_{-1}(\aA_1)}=\|\aA_1^{-1}x\|_{\xX_1(\aA)}=\|\aA \aA_{1}^{-1}x\|=\|x\|,
\] 
which can be extended by density for all $x\in \xX_0$, showing also the equality of the spaces $\xX_{-1}({\aA_{1}})=\xX_0$ (with the same norm).

\medskip\noindent (c) By construction the operator $(\aA_1)_{-1}:\xX_1(A)\to \xX_{-1}(\aA_1)$ is the unique continuous  extension of
\[
\aA_1:\dom(\aA_1)=\dom(\aA^2)\to \xX_1(A),
\]
and $(\aA_1)_{-1}$ is an isometric isomorphism. For $x\in \xX_1(A)$ we have $\|x\|_{\xX_{-1}(A_1)}=\|\aA_{1}^{-1}x\|_{\xX_1(A)}=\|x\|$. But then it follows that $(\aA_1)_{-1}=\aA:\dom(\aA)\to \xX_0$.

\medskip\noindent (d) The space $X_{-1}(A_{1})$ is defined by
\[
X_{-1}(A_{1}):=(\underline{A_1})_{-2}(X_1(A))={((\underline{A_1})_{-1})}_{-1}(X_1(A))=\aA_{-1}(X_1(A))=A X_1(A)=X_0,
\]
by part (c). For the norm equality let $x\in X_0$. Then
\[
\|x\|=\|AA^{-1}x\|=\|A^{-1}x\|_{X_1(A)}=\|\aA_{-1}^{-1}x\|_{X_1(A)}=\|(\aA_{1})_{-2}^{-1}x\|_{X_1(A)}=\|x\|_{X_{-1}(A_1)}.
\]
The last assertion is equally easy to prove: $(A_1)_{-1}=(\aA_1)_{-2}|_{X_1(A)}=A$.
\end{proof}

Recall the standing assumption that $\aA=A\rest_{\xX_0}$ is densely defined in $\xX_0=\overline(\dom(A))$.
The following proposition plays the key role for the extension of operators on the extrapolation spaces, particularly for the construction of extrapolated semigroups in Section \ref{sec:sgrps}.
\begin{proposition} \label{prop:extT}
\begin{abc}
\item Let $n\in \NN$.If $T\in\LLL(X_0)$ is a linear operator commuting with $A^{-1}$, then the operator $T$ has a unique continuous extension to $\xX_{-n}$ denoted by $\tT_{-n}$. The operator $\tT_{-n}$ is the restriction of $\tT_{-n-1}$. $X_{-n}$ is invariant under $\tT_{-n-1}$, its restriction is denoted by $T_{-n}$, for which $T_{-n}\in \LLL(X_{-n})$. For $k,n\in -\NN$ the operators $\tT_n$, $\tT_k$ are all similar; the same holds for $T_n$ and $T_k$.
\item Let $\tT\in\LLL(\xX_0)$ such that it leaves $\dom(A)$ invariant and it commutes with $\aA^{-1}=A^{-1}\rest_{X_0}$. Then $\tT_{-1}x=A\tT A^{-1}x$ for each $x\in X_0$, and as a consequence, $\tT_{-1}:\xX_{-1}\to \xX_{-1}$ leaves $X_0$ invariant (and, of course, extends $\tT$).
\end{abc}
\end{proposition}

\begin{proof}(a)
For $x\in X_0$ we have $\|Tx\|_{X_{-1}(A)}=\|A^{-1}Tx\|=\|TA^{-1}x\|\leq \|T\|\cdot \|A^{-1}x\|=\|T\|\cdot \|x\|_{X_{-1}(A)}$. So that $T:(X_0,\|\cdot\|_{X_{-1}(A)})\to (X_0,\|\cdot\|_{X_{-1}(A)})$ is continuous, and hence has a unique continuous extension $\tT_{-1}$ to $\xX_{-1}$. This extension commutes with $\aA^{-1}_{-1}$, because $T$ commutes with $A^{-1}$ and $\aA_{-1}^{-1}$ is the unique continuous extension of $A^{-1}$. By iteration we obtain the continuous extensions $T_{-n}$ onto $\xX_{-n}$, which then all commute with the corresponding $\aA_{-n}^{-1}$. By construction $\tT_{-n}$ is a restriction of $\tT_{-n-1}$. We prove that $X_{-1}$ is invariant under $\tT_{-2}$. Let $x\in X_{-1}$, hence $x=\aA_{-2}y$ for some $y\in X_0$. Then $Ty=\tT_{-2}y=\tT_{-2}\aA_{-2}^{-1}x=\aA_{-2}^{-1}\tT_{-2}x$, hence $\tT_{-2}x=\aA_{-2}Ty\in X_{-1}$, i.e., the invariance of $X_{-1}$ is proved. We have for $x\in X_{-1}$ that $\|T_{-1}x\|_{X_{-1}}=\|A_{-2}^{-1}T_{-1}x\|=\|A_{-2}^{-1}\tT_{-2}x\|=\|\tT_{-2}A_{-2}^{-1}x\|\leq \|\tT_{2}\|\cdot \|\aA_{-2}^{-1}x\|=\|\tT_{2}\|\cdot \|x\|_{X_{-1}}$, therefore $T_{-1}\in \LLL(X_{-1})$. The assertion about $T_{-n}$ follows by recursion.

\medskip It is enough to prove the similarity of $T_{0}=T$ and $T_{-1}$, and the similarity of $\tT_{0}$ and $\tT_{-1}$. The latter assertions can be proved as follows: For $x\in \dom(A)$ we have
\[
\aA_{-1}^{-1}\tT_{-1}\aA_{-1}x=\aA_{-1}^{-1}\tT_{-1}Ax=\aA_{-1}^{-1}TAx=\aA_{-1}^{-1}ATx=\aA_{-1}^{-1}A_{-1}Tx=\tT x,
\]
then by continuity and denseness the equality follows even for $x\in \xX_0$. For the similarity of $T$ and $T_{-1}$ take $x\in X_0$. Then 
\[
A_{-1}^{-1}T_{-1}A_{-1}x=\aA_{-2}^{-1}\tT_{-2}\aA_{-2}x=\tT_{-1}x=T x.
\]

\medskip\noindent (b) Let $x\in X_0\subseteq \xX_{-1}$. Then there is a sequence $(x_n)$ in $\xX_0$ with $x_n\to x$ in $\xX_{-1}$ (see Proposition \ref{prop:altercompl}). But then $A^{-1}x_n\to A^{-1}x$ in $\xX_0$ and $\tT x_n\to \tT_{-1}x$ in $\xX_{-1}$ by part (a). These imply $\tT A^{-1}x_n=A^{-1}\tT x_n\to \aA^{-1}_{-1}\tT_{-1}x$. Hence we conclude $\tT A^{-1}x=\aA^{-1}_{-1}\tT_{-1}x$ and $A\tT A^{-1}x=\tT_{-1}x$ for $x\in X_0$.
\end{proof}

Haase in \cite{Haase} and Wegner in \cite{W} have constructed the so-called \emph{universal extrapolation space} $\xX_{-\infty}$ as follows: Suppose $A$ is densely defined (this assumption is \emph{not} posed by Haase), then $X_n=\xX_n$ for each $n\in \ZZ$ and let $X_{-\infty}$ to be the inductive limit of the sequence of Banach spaces $(X_{-n})_{n\in\NN}$ (algebraic inductive limit in \cite{Haase}).
One can extend the operator $A$ to an operator $A_{-\infty}:X_{\infty}\to X_\infty$, for which
\[
A_{-\infty}\rest_{X_n}=A_n,\quad n\in\mathbb{Z}.
\]
We now look at a converse situation, and our starting point is the following: Let $\D$ be a locally convex space and suppose that we have a continuous operator $\A:\D\to\D$ such that we can embed the Banach space $X_0$ continuously in $\D$, i.e., there is a continuous injective map $i:X_0\to\D$, and so we can identify $X_0$ with a subspace of $\D$. We also assume that $\lambda-\A:i(X_0)\to\D$ is injective and that
\[
\dom(A)=\{x\in X_0:\ \A\circ i(x)\in i(X_0)\},
\]
and
\[
i\circ A=\A\circ i|_{\dom(A)}.
\]
As a matter of fact, this setting can be also used to construct the extrapolation spaces $\xX_{-n}$, $X_{-n}$ for $n\in \NN$ similarly to our Definition \ref{def:extra}, as indicated by the next theorem, which is proved here based on the results in Section \ref{sec:invert}. Notice that we do not assume that $A$ is a Hille--Yosida operator or densely defined.

\begin{theorem}\label{thm:iden}
Let $X_0$ be a Banach space with a continuous embedding $i:X_0\to\D$ into a locally convex space $\D$, let $A:\dom(A)\to X_0$ be a linear operator with $\lambda\in\rho(A)$ such that $A=\A_{|X_0}$ (after identifying $X_0$ with a subspace of $\D$ as described above). We suppose furthermore that $\lambda-\A$ is injective on $X_0$. There is a continuous embedding $i_{-1}:X_{-1}\to\D$ which extends $i$. After identifying $X_{-1}$ with a subspace of $\D$ (under $i_{-1}$) we have
\begin{align*}
X_{-1}=\{(\lambda-\A)x:\ x\in X_0\},\quad \underline{X}_{-1}=\{(\lambda-\A)x:\ x\in\underline{X_0}\}\quad \text{and}\quad A_{-1}=\A\rest_{X_{-1}}.
\end{align*}
\end{theorem}

\begin{proof}
Without lost of generality we may assume that $\lambda=0$. Recall that ${A_{-1}}\rest_{X_0}=A$ and $A_{-1}$ is an isometric isomorphism $A_{-1}:X_0\to X_{-1}$. We now define the embedding $i_{-1}:X_{-1}\to\D$ by
\[
i_{-1}:=\A\circ i\circ A_{-1}^{-1},
\]
which is indeed injective and continuous by assumption. Of course, $i_{-1}$ extends $i$ since we have $i=\A\circ i\circ A^{-1}$. We can write
\[
i_{-1}\circ A_{-1}=\A\circ i\circ A_{-1}^{-1}\circ A_{-1}=\A\circ i,
\]
which yields the following commutative diagram:
\begin{align*}
\xymatrix 
{
X_0\ar[rr]^{i}\ar[dd]^{A_{-1}} & & \D\ar[dd]_{\A} \\
 & &\\
X_{-1}\ar[rr]_{i_{-1}} & & \D
}
\end{align*}
All the assertions follow from this.
\end{proof}

The last corollary in this s can be proved by recursion based on the facts proved in Section \ref{sec:invert}.

\begin{corollary}
Let $\A$, $X_0$ and $\D$ as in Theorem \ref{thm:iden}. Then $X_{n}\subseteq\D$ and $A_{n}=\A\rest_{X_{-n}}$ for each $n\in\ZZ$ (identifying $X_n$ under an embedding $i_n$ with a subspace of $\D$).
\end{corollary}

\section{Intermediate spaces for operators with rays of minimal growth}\label{sec:minimal}
The following definition of intermediate, and as a matter of fact interpolation spaces, just as well many results in this section are standard, and we refer, e.g., to the book of Lunardi \cite[Chapter 3]{Lunardi}, and to Engel, Nagel \cite[Section II.5]{EN} for the case of semigroup generators.
\begin{definition}\label{def:Fav}
Let $A$ be a linear operator on the Banach space $X_0$ with a ray of minimal growth, i.e., suppose that $(0,
\infty)\subseteq\rho(A)$ and for some $M\geq 0$
\begin{equation}\label{eq:weakHY}
\|\lambda R(\lambda,A)\|\leq M\quad\text{for all $\lambda>0$}.
\end{equation}
For $\alpha\in\left(0,1\right]$ and $x\in X_0$ we define
\[
\|x\|_{\Fav_\alpha(A)}:=\sup_{\lambda>0}\|\lambda^\alpha AR(\lambda,A)x\|,
\]
and the abstract Favard space of order $\alpha$
\[
\Fav_\alpha(A):=\bigl\{x\in X_0:\|x\|_{\Fav_\alpha(A)}<\infty\bigr\}.
\]
In the literature the notation $D_A(\alpha,\infty)$ is more common, see, e.g., \cite{Lunardi}, but for notational convenience we stick to our notation in this paper.
We further set
\[
\Fav_0(A)=\Fav_1(A_{-1}),
\]
see \cite[Section II.5(b)]{EN} for the case of semigroup generators. 
\end{definition}
The standing assumption in this section will be that $A$ satisfies \eqref{eq:weakHY}.
\begin{proposition}
\begin{abc}
\item
The Favard space $\Fav_\alpha(A)$ becomes a Banach space if endowed with the norm $\|\cdot\|_{\Fav_\alpha(A)}$.
\item $X_0$ is isomorphic to a closed subspace of $\Fav_0(A)$.
\end{abc}
\end{proposition}
For the case when $A$ is a Hille--Yosida operator, The statement that $X_0$ is closed subspace of $\Fav_0(A)$ is due to Nagel and Sinestrari \cite[Proof of Prop. 2.7]{NS} 

\begin{proof}
(a) is an easy checking of properties.

\medskip\noindent (b) For $x\in X_0$ we have
\[\|\lambda A_{-1}R(\lambda,A_{-1})x\|_{X_{-1}(A)}=\|\lambda A R(\lambda,A)x\|_{X_{-1}(A)}=\|\lambda A^{-1}A R(\lambda,A)x\|\leq M\|x\|,\]
yielding 
\[
\|x\|_{\Fav_0(A)}=\|x\|_{\Fav_1(A_{-1})}\leq M\|x\|.
\]
On the other hand, since $A$ and $A_{-1}$ are similar, we have $\sup_{\lambda>0}\|\lambda R(\lambda,A_{-1})\|\leq M'$ for some $M'\geq0$ and for all $\lambda>0$. In particular, by Remark \ref{rem:2}, $\lambda R(\lambda, A_{-1})x\to x$ for each $x\in \xX_{-1}$. From which we obtain for $x\in X_0$ that
\begin{align*}
\|x\|&=\|A_{-1}x\|_{X_{-1}(A)}=\Bigl\|\lim_{\lambda\to 0} \lambda R(\lambda,A_{-1})A_{-1}x\Bigr\|_{X_{-1}(A)}\leq \sup_{\lambda>0}\|\lambda A_{-1}R(\lambda,A_{-1})x\|_{X_{-1}(A)}\\
&=\|x\|_{\Fav_1(A_{-1})}=\|x\|_{\Fav_0(A)},
\end{align*}
showing the equivalence of norms $\|\cdot\|$ and $\|x\|_{\Fav_0(A)}$ on $X_0$.
\end{proof}

We will also need the following well-known result, see, e.g., \cite[Chapters 1 and 3]{Lunardi}, for which we give a direct, short proof.
\begin{proposition}
For $\alpha\in (0,1]$ we have $\Fav_\alpha(A)\subseteq \overline{\dom(A)}=\xX_0$.
\end{proposition}
\begin{proof}
We have 
\[
AR(\lambda,A)x=\lambda R(\lambda,A)x-x,
\]
so that
\[
\|\lambda R(\lambda,A)x-x\|\leq \frac{\|x\|_{\Fav_\alpha(A)}}{\lambda^\alpha}\to 0\quad \text{as $\lambda\to \infty$.}
\]
\end{proof}

\begin{definition}\label{def:XHol}
Let $A$ be a linear operator on the Banach space $X_0$ satisfying \eqref{eq:weakHY}. 
For $\alpha\in (0,1)$ we set 
\begin{align*}
\XHol_\alpha(A)&:=\Bigl\{x\in \Fav_\alpha(A):\lim_{\lambda\to\infty} \lambda^\alpha AR(\lambda,A)x=0\Bigr\},\\
\intertext{and we recall from Section \ref{sec:invert}}
\XHol_0(A)&:=\overline{\dom(A)},\quad \XHol_1(A)=\dom(A|_{\XHol_0(A)}).
\end{align*}
\end{definition}

The proof of the next proposition is straightforward, but also well-known.

\begin{proposition}
For $\alpha,\beta\in (0,1)$ with $\alpha>\beta$ we have 
\[
 \XHol_1(A)\inj \XHol_\alpha(A)\subseteq \Fav_\alpha(A) \inj \XHol_\beta(A)\subseteq \Fav_\beta(A)\inj \XHol_0(A)\subseteq X_0(A)
\]
with $\inj$ denoting continuous and dense embeddings of Banach spaces, and $\subseteq$ denoting inclusion of closed subspaces.
\end{proposition}
\begin{proof}
For $x\in \Fav_\alpha(A)$ we have
\[
\|\lambda^{\beta} AR(\lambda,A)x\|=\lambda^{\beta-\alpha}\|\lambda^{\alpha} AR(\lambda,A)x\|\leq \lambda^{\beta-\alpha}\|x\|_\alpha\to 0\quad\text{as $\lambda\to\infty$},
\]
and this proves also the continuity of $\Fav_\alpha(A)\inj \XHol_\beta(A)$.
The other statements can be proved by similar reasonings.
\end{proof}

\begin{proposition}
 \begin{abc}
 \item 
The spaces $\Fav_\alpha(A)$ and $\XHol_\alpha(A)$ are invariant under each $T\in\LLL(X_0)$ which commutes with $A^{-1}$.
\item Let $T\in\LLL(X_0)$ be commuting with $A^{-1}$. The space $\Fav_0(A)$ is invariant under $T_{-1}$.
\end{abc}
 \end{proposition}
\begin{proof}
(a)
Suppose that $T\in\LLL(X_0)$ commutes with $R(\cdot, A)$ and let $x\in \XHol_\alpha(A)$. We have to show that $Tx\in\XHol_\alpha(A)$. Since $T$ is assumed to be bounded we make the following observation:
$$\|\lambda^{\alpha}AR(\lambda,A)Tx\|=\|\lambda^{\alpha}ASR(\lambda,A)x\|\leq\|T\|\cdot\|\lambda^{\alpha}AR(\lambda,A)x\|.$$
This implies both assertions.

\medskip\noindent (b) Follows from part (a) applied to $T_{-1}$ on the space $X_{-1}$.
\end{proof}

\begin{definition}
For $\alpha\in \RR$ we write $\alpha=m+\beta$ with $m\in \ZZ$ and $\beta\in (0,1]$, and define
\[
\Fav_\alpha(A):=\Fav_{\beta}(A_m),
\]
with the corresponding norms.
For $\alpha\not\in\ZZ$ we define
\[
\XHol_\alpha(A):=\XHol_{\beta}(A_m),
\]
also with the corresponding norms.
\end{definition}
In particular we have for $\alpha\in (0,1)$ that
\[
\XHol_{-\alpha}(A)=\XHol_{1-\alpha}(A_{-1})\quad \text{and}\quad \Fav_{-\alpha}(A)=\Fav_{1-\alpha}(A_{-1}).
\]
This definition is consistent with Definitions \ref{def:Fav} and \ref{def:XHol}. The following property of these spaces can be directly deduced from the definitions and the previous assertions (by using recursion):
\begin{proposition}
For any $\alpha,\beta\in \RR$ with $\alpha>\beta$ we have
\[
\XHol_\alpha(A)\subseteq \Fav_\alpha(A) \inj \XHol_\beta(A)\subseteq \Fav_\beta(A)
\]
with $\inj$ denoting continuous and dense embeddings of Banach spaces, and $\subseteq$ denoting inclusion of closed subspaces.
\end{proposition}

Now we put these spaces in a more general context presented at the end of Section \ref{sec:invert}.

\begin{proposition}\label{cor:ExtFav}
\begin{abc}
\item For $\alpha\in\left(0,1\right]$ we have $A_{-1}F_{\alpha}=F_{\alpha-1}$ and $A_{-1}\xX_{\alpha}=\xX_{\alpha-1}$.
\item
For $\alpha\in\left(0,1\right]$ and $\mathcal{A}$, $\lambda$ and $\D$ as in Theorem \ref{thm:iden} we have
\[
F_{-\alpha}=\Bigl\{(\lambda-\mathcal{A})y\in X_{-1}:\ y\in F_{1-\alpha}\Bigl\}.
\]
If $\alpha\in\left(0,1\right)$, then
\begin{align*}
\xX_{-\alpha}&=\Bigl\{(\lambda-\mathcal{A})y\in X_{-1}:\ y\in \xX_{1-\alpha}\Bigl\}.
\end{align*}
\end{abc}
\end{proposition}

\section{Semigroup generators}\label{sec:sgrps}
In this section we consider intermediate and extrapolation spaces when the linear operator $A:\dom(A)\to X_0$ is the generator of a semigroup $(T(t))_{t\geq 0}$ (meaning that $T:[0,\infty)\to \LLL(X_0)$ is a monoid homomorphism) in the following sense:
\begin{assumption}\label{asp:sgrps}\begin{num}

\item Let $X_0$ be a Banach space, and let $Y\subseteq X_0'$ be a norming subspace, i.e.,
\[
\|x\|=\sup_{y\in Y, \|y\|\leq1}|\dprod{ x}{y}|\quad\text{for each $x\in X_0$}.
\]
\item Let $T:[0,\infty)\to \LLL(X_0)$ be a semigroup of contractions, which is not necessarily supposed to be strongly continuous, but for which a generator $A:\dom(A)\to X_0$ exists in the sense that
\begin{equation}\label{eq:laplace}
R(\lambda, A)x=\int_0^\infty \ee^{-\lambda s}T(s)x\ \dd s
\end{equation}
exists for each $\lambda\geq 0$ as a weak integral with respect to the dual pair $(X_0,Y)$, i.e., for each $y\in Y$ and $x\in X_0$
\[
\dprod{R(\lambda,A)x}{y}=\int_0^\infty\ee^{-\lambda s}\dprod{T(s)x}{y}\ \dd s,
\]
 and $R(\lambda,A)\in\LLL(X_0)$ is in fact the resolvent of a linear operator $A$. 
\item We also suppose that $T(t)$ commutes with $A^{-1}$ for each $t\geq 0$.
\end{num}
\end{assumption}

If the semigroup $(T(t))_{t>0}$ is only exponentially bounded of type $(M,\omega)$, that is
\[
\|T(t)\|\leq M\ee^{\omega t}\quad\text{for all $t\geq0$,}
\]
then one rescale it (consider$(\ee^{-(\omega+1) t}T(t))_{t\geq 0}$), and renorm the space such that the rescaled semigroup becomes a contraction semigroup. Moreover, the new semigroup has negative growth bound, meaning that $T(t)\to 0$ in norm exponentially fast as $t\to \infty$, and has an invertible generator.

\begin{remark}
\begin{abc}
	\item There are several important classes of semigroups, whose elements satisfy Assumption \ref{asp:sgrps}, hence can be treated in a unified manner: $\pi$-semigroups of Priola \cite{Priola}, weakly continuous semigroups of Cerrai \cite{Cerrai}, bi-continuous semigroups of K\"uhnemund. We will concentrate on this latter class of semigroups in Section \ref{sec:bicont}. In this framework Kunze \cite{Kunze2009} introduced the notion of integrable semigroups, which we briefly describe next.

\item Since we have
\[
\|y\|=\sup_{x\in X_0, \|x\|\leq1}{|\langle x,y\rangle|}
\]
and by the norming assumption
\[
\|x\|=\sup_{y\in Y, \|y\|\leq1}{|\langle x,y\rangle|},
\]
the pair $(X_0,Y)$ is called a norming dual pair. Kunze has worked out the theory of semigroups on such norming dual pairs in \cite{Kunze2009}, we recall at least the basic definitions here: We assume without loss of generality that $Y$ is a Banach space and consider the weak topology $\sigma(X_0,Y)$ on $X_0$.
 An \emph{integrable semigroup} on the pair $(X_0,Y)$ is a semigroup $(T(t))_{t\geq0}$ of $\sigma$-continuous linear operators of type $(M,\omega)$ such that:
\begin{num}
\item $(T(t))_{t\geq0}$ is a semigroup, i.e. $T(t+s)=T(t)T(s)$ and $T(0)=\Id$ for all $t,s\geq0$.
\item $(T(t))_{t\geq0}$ is exponentially bounded, i.e. there exists $M\geq1$ and $\omega\in\mathbb{R}$ such that $\|T(t)\|\leq M\ee^{\omega t}$ for all $t\geq0$. We then say that $(T(t))_{t\geq0}$ is of type $(M,\omega)$.
\item For all $\lambda$ with $\text{Re}(\lambda)>\omega$, there exists an $\sigma$-continuous linear operator $R(\lambda)$ such that for all $x\in X_0$ and all $y\in Y$
\[ 
\dprod {R(\lambda)x}{y}=\int_0^{\infty}\ee^{-\lambda t}\dprod{T(t)x}{y} \ \dd t.
\]
\end{num}
Kunze defines the generator $A$ of the semigroup as the (unique) operator $A:\dom(A)\to X_0$ (if it exists at all) with $R(\lambda)=(\lambda-A)^{-1}$, precisely as we did in Assumption \ref{asp:sgrps}. Note that $\sigma$-continuity of $T(t)$ can be used to assure that $Y$ is invariant under $T'(t)$, cf.{} the next remark.
\end{abc}
\end{remark}

Some further consequences of the previous assumptions follow:

\begin{remark} The commutation property can be verified easily if $Y$ can be chosen such that it is invariant under $T'(t)$ for each $t\geq 0$:
\begin{align*}
\dprod{A^{-1}T(t)x}{y}&=\int_0^{\infty}\dprod{T(s)T(t)x}{y}\ \dd s=\int_0^{\infty}\dprod{T(s+t)x}{y}\ \dd s\\
&=\Dprod{\int_0^{\infty}T(s)x\ \dd s}{T'(t)y}=\dprod{T(t)A^{-1}x}{y}.\end{align*}
\end{remark}
\begin{remark}
\begin{num}
\item From \eqref{eq:laplace} it follows that for each $x\in X_0$
\begin{equation}\label{eq:midnight}
T(t)x-x=A\int_0^t T(s)x\ \dd s.
\end{equation}
Indeed, we have by \eqref{eq:laplace} that
\begin{align*}
x&=A\int_0^\infty T(s)x\ \dd s\\
T(t)x&=A\int_0^\infty T(s)T(t)x\ \dd s=\int_t^\infty T(s)x\ \dd s.
\end{align*}
Subtracting the first of these equation from the second one, we obtain the statement.
\item If moreover $A$ commutes with $T(t)$ for each $t\geq0$, then for each $x\in\dom(A)$ we have
\begin{equation}\label{eq:midnightDA}
T(t)x-x=\int_0^t T(s)Ax\ \dd s.
\end{equation}
Indeed, as in (1) we have by \eqref{eq:laplace}
\begin{align*}
-x&=-A^{-1}Ax=\int_0^\infty T(s)Ax\ \dd s\\
-T(t)x&=-A^{-1}T(t)Ax=\int_0^\infty T(s)T(t)Ax\ \dd s=\int_t^\infty T(s)Ax\ \dd s.
\end{align*}
By a simple subtraction we obtain the statement.
\end{num}
\end{remark}

The next lemma and its proof is standard for various classes of semigroups.
\begin{lemma}\label{lem:spacecont}
Let $(T(t))_{t\geq 0}$ is (locally) norm bounded, then 
\[
X_{\conti}:=\{x\in X_0: t\mapsto T(t)x\text{ is}\ \|\cdot\|\text{-continuous}\bigr\}
\]
is a closed a subspace of $X_0$ invariant under the semigroup.
Under Assumption \ref{asp:sgrps} we have
\[
\xX_0=\overline{\dom(A)}=X_{\conti}.
\]
\end{lemma}
\begin{proof}
The closedness and invariance of $X_{\conti}$ are evident. We first show $\dom(A)\subseteq X_{\conti}$, which implies $\overline{\dom(A)}\subseteq X_{\conti}$ by closedness of $X_{\conti}$. Let $x\in\dom(A)$. By \eqref{eq:midnightDA} we conclude that $T(t)x-x=\int_0^t{T(s)Ax\ \dd s}$. As $t\to0$ the integral here tends to zero in $\|\cdot\|$: 
\begin{align*}
\|T(t)x-x\|=\sup_{\|y\|\leq1}|\dprod{T(t)x-x}{y}|\leq \sup_{\|y\|\leq1}\int_0^t|\dprod{T(s)Ax}{y}|\ \dd s\leq t\|Ax\|.
\end{align*} 
Whence we conclude that $\dom(A)\subseteq X_{\conti}$. For the converse inclusion suppose that $x\in X_{\conti}$. Again by \eqref{eq:midnight} we obtain that the sequence of vectors $x_n:=n\int_0^{\frac{1}{n}}{T(s)x\ \dd s}\in \dom(A)$ ($n\in\NN$) converges to $x$. Indeed:
\begin{align*}
\|x_n-x\|=\sup_{\|y\|\leq1}|\dprod{ x_n-x}{y}|\leq\sup_{\|y\|\leq1}n\int_0^{\frac{1}{n}}|\dprod {T(s)x-x}{y}|\ \dd s\leq n\int_0^{\frac{1}{n}}{\|T(s)x-x\|\ \dd s}.
\end{align*}
 By the continuity of $s\mapsto T(s)x$ we obtain the inclusion $X_{\conti}\subseteq\overline{\dom(A)}$.

\end{proof}

Based on this lemma one can prove the following characterization of the Favard and H\"older spaces:
 \begin{proposition}\label{prop:FavCont}
 Let $(T(t))_{t\geq0}$ be a semigroup satisfying Assumption \ref{asp:sgrps} with negative growth bound and generator $A$. For
 $\alpha\in(0,1]$ define
\begin{equation}\label{equ:Fav}
\TFav_{\alpha}(T):=\Bigl\{x\in X_0:\sup_{s>0}\frac{\|T(s)x-x\|}{s^{\alpha}}<\infty\Bigr\}=\Bigl\{x\in X_0:\sup_{s\in(0,1)}\frac{\|T(s)x-x\|}{s^{\alpha}}<\infty\Bigr\},
\end{equation}
and for $\alpha\in (0,1)$ define
\begin{align}\label{equ:Hol}
 \THol_{\alpha}(T)&:=\Bigl\{x\in X_0:\sup_{s>0}\frac{\|T(s)x-x\|}{s^{\alpha}}<\infty\text{ and }\lim_{s\downto 0}\frac{\|T(s)x-x\|}{s^{\alpha}}=0\Bigr\}\\
\notag &=\Bigl\{x\in X_0:\lim_{s\downto 0}\frac{\|T(s)x-x\|}{s^{\alpha}}=0\Bigr\},
\end{align}
which become Banach spaces if endowed with the norm
\[
\|x\|_{\alpha}:=\sup_{s>0}\frac{\|T(s)x-x\|}{s^{\alpha}}.
\]
The space $ \THol_{\alpha}(T)$ is a closed subspace of $\TFav_\alpha(T)$. These spaces are  invariant under the semigroup $(T(t))_{t\geq0}$, and $\THol_\alpha(T)$ is the space of $\|\cdot\|_\alpha$-strong continuity in $\TFav_\alpha(T)$.
For $\alpha\in (0,1]$ we have $\Fav_\alpha(A)=\TFav_\alpha(T)$ and for $\alpha\in (0,1)$ we have $\XHol_\alpha(A)=\THol_\alpha(T)$ with equivalent norms.
\end{proposition}
\begin{proof}
The invariance of the spaces $\TFav_\alpha$ can be proved as follows: For $x\in\TFav_{\alpha}$ we have:
\begin{align*}
\|T(t) x\|_{{\alpha}}&=\sup_{s>0}{\frac{\|T(s)T(t)x-T(t)x\|}{s^{\alpha}}}\leq\|T(t)\|\cdot\sup_{s>0}{\frac{\|T(s)x-x\|}{s^{\alpha}}}\leq M\|x\|_{{\alpha}}.
\end{align*}
Similar reasoning proves the invariance of $\THol_\alpha$. Since $\TFav_\alpha\subseteq X_{\conti}=\xX_0=\overline{\dom(A)}$ and $\Fav_\alpha(A)\subseteq \xX_0=\overline{\dom(A)}$, the rest of the assertions follow from the corresponding results concerning $C_0$-semigroups, see, e.g., \cite[Sec. II.5]{EN}. 
\end{proof}

To complete the picture we recall the next result from \cite[Chapter 5]{Lunardi}, which is formulated there only for $C_0$-semigroups as a theorem, but Lunardi also remarks, without stating the precise assumptions, that this result still holds if one omits the strong continuity assumption. We require here the conditions from Assumption \ref{asp:sgrps}, under which the proof is verbatim the same as for the $C_0$-case, and is based on formulas \eqref{eq:midnight} and \eqref{eq:midnightDA}.

\begin{proposition}Let $A$ generate the semigroup $(T(t))_{t>0}$ of negative growth bound in the sense describe in Assumption \ref{asp:sgrps}. Then for $p\in [1,\infty]$ and $\alpha\in (0,1)$ we have for the interpolation space:
\[
(X,\dom(A))_{\alpha,p}=\{x\in X:\ t\mapsto\psi(t):=t^{-\alpha}\|T(t)x-x\|\in \Ell^p_*(0,\infty)\},
\]
where $\Ell^p_*(0,\infty)$ denotes the $\Ell^p$-space with respect to the Haar measure $\frac {\dd t}{t}$ on the multiplicative group $(0,\infty)$. Moreover,
   the norms $\|x\|_{\alpha,p}$ and 
\[
\|x\|^{**}_{\alpha,p}=\|x\|+\|\psi\|_{\Ell^p_*(0,\infty)}
\]
are equivalent.
\end{proposition}

We conclude this section with the construction of the extrapolated semigroup as a direct consequence of the results in Section \ref{sec:invert}, particularly of Proposition \ref{prop:extT}.
\begin{proposition}
Let $A$ generate the semigroup $(T(t))_{t\geq 0}$ of negative growth bound in the sense of Assumption \ref{asp:sgrps}. Then there is an extension $(T_{-1}(t))_{t\geq 0}$ of the semigroup $(T(t))_{t\geq 0}$ on the extrapolated space $X_{-1}$, whose generator is $A_{-1}$.
\end{proposition}

\section{Bi-continuous semigroups}\label{sec:bicont}

In this section we concentrate on extrapolation spaces for generators of \emph{bi-continuous semigroups}. Such semigroups were introduced by K\"uhnemund in \cite{Ku} and possess generators as described in Section\ref{sec:sgrps}. The following assumptions, as proposed by K\"uhnemund, will be made during the whole section.

\begin{assumption}\label{asp:bicontspace} 
Consider a triple $(X_0,\|\cdot\|,\tau)$ where $X_0$ is a Banach space, 
\begin{num}
\item $\tau$ is a locally convex Hausdorff topology coarser than the norm-topology on $X_0$, i.e. the identity map $(X_0,\|\cdot\|)\to(X_0,\tau)$ is continuous;
\item $\tau$ is sequentially complete on the closed unit ball;
\item The dual space of $(X_0,\tau)$ is norming for $X_0$, i.e.,
\begin{equation}\label{eq:norm}
\|x\|=\sup_{\substack{\varphi\in(X_0,\tau)'\\\|\varphi\|\leq1}}{|\varphi(x)|}.\end{equation}
\end{num}
\end{assumption}

\begin{remark}
\label{rem:seminorm}\label{rem:pnorming}
\begin{num}
\item
There is a related notion of so-called Saks spaces, see \cite{CooperSaks}. By definition a \emph{Saks space} is a triple $(X_0,\|\cdot\|,\tau)$ such that $X_0$ is a vector space with a norm $\|\cdot\|$ and locally convex topology $\tau$ in such a way that $\tau$ is weaker than the $\|\cdot\|$-topology, but the closed unit ball is $\tau$-complete. In particular, $X_0$ is a Banach space.
\item There is also a connection to the the norming dual pairs we discussed in the previous section. In particular, $(X_0,Y)$ with $Y=(X_0,\tau)'$ is a norming dual pair.
\item Kraaij puts this setting in a more general framework of locally convex spaces with mixed topologies, see \cite[Sec.{} 4]{Kraaij2016}, and also \cite[App.{} A]{FaPHD}

\item Assumption \eqref{eq:norm} is equivalent to the following: There is a set $\semis$ of $\tau$-continuous seminorms defining the topology $\tau$, such that
\begin{equation}\label{eq:semisnorm}
\|x\|=\sup_{p\in\semis}p(x).
\end{equation}
This description is also used by Kraaij in \cite{Kraaij2016}, cf.{} his Lemma 4.4. Note also that by this remark and by Lemma 3.1 in \cite{CooperSaks} we see that a Saks space satisfies Assumption \ref{asp:bicontspace}.
Indeed, assume \eqref{eq:norm} and let $\semis$ be the collection of \emph{all} $\tau$-continuous seminorms $p$ such that $p(x)\leq \|x\|$. Then $|\varphi(\cdot)|\in \semis$ for each $\varphi\in(X_0,\tau)'$ with $\|\varphi\|\leq 1$, and \eqref{eq:semisnorm} is trivially satisfied. If $q$ is any $\tau$-continuous seminorm, then $q(x)\leq M\|x\|$ for some constant $M$ and for all $x\in X_0$. So that $q/M\in \semis$, proving that $\semis$ defines precisely the topology $\tau$. For the converse implication suppose that \eqref{eq:semisnorm} holds. Then by the application of the Hahn--Banach theorem we obtain \eqref{eq:norm}. 
\end{num}
\end{remark}

Now we are in state to formulate the definition of a bi-continuous semigroup.

\begin{definition}[K\"uhnemund \cite{Ku}]\label{def:bicontsemi}
Let $X_0$ be a Banach space with norm $\|\cdot\|$ together with a locally convex topology $\tau$, such that conditions in Assumption \ref{asp:bicontspace} are satisfied. We call $(T(t))_{t\geq0}$ a \emph{bi-continuous semigroup} if
\begin{num}
\item $ T(t+s)=T(t)T(s)$ and $T(0)=I$ for all $s,t\geq 0$.
\item $(T(t))_{t\geq0}$ is strongly $\tau$-continuous, i.e. the map $\varphi_x:[0,\infty)\to(X_0,\tau)$ defined by $\varphi_x(t)=T(t)x$ is continuous for every $x\in X_0$.
\item $(T(t))_{t\geq0}$ is exponentially bounded, i.e., has type $(M,\omega)$ for some $M\geq 1$ and $\omega\in \RR$.
\item $(T(t))_{t\geq0}$ is locally-bi-equicontinuous, i.e., if $(x_n)_{n\in\NN}$ is a norm-bounded sequence in $X_0$ which is $\tau$-convergent to $0$, then also $(T(s)x_n)_{n\in\NN}$ is $\tau$-convergent to $0$ uniformly for $s\in[0,t_0]$ for each fixed $t_0\geq0$.
\end{num}
\end{definition}

As in the case of $C_0$-semigroups we can define a generator for a bi-continuous semigroup in the following way:

\begin{definition}Let $(T(t))_{t\geq0}$ be a bi-continuous semigroup on $X_0$. The generator $A$ is defined by
\[Ax:=\tlim_{t\to0}{\frac{T_0(t)x-x}{t}}\] with the domain 
\[\dom(A):=\Bigl\{x\in X_0:\ \tlim_{t\to0}{\frac{T(t)x-x}{t}}\ \text{exists and} \ \sup_{t\in(0,1]}{\frac{\|T(t)x-x\|}{t}}<\infty\Bigr\}.\]
\end{definition}

This definition of the generator leads to a couple of important properties and important examples in this context are evolution semigroups on $\BC(\RR,X)$, semigroups induced by flows, adjoint semigroups and the Ornstein--Uhlenbeck semigroup on $\BC(\mathcal{H})$. 

The following theorem sums up some properties of bi-continuous semigroups and their generators (see \cite{Ku},\cite{FaStud}):

\begin{theorem}
Let $(T(t))_{t\geq0}$ be a bi-continuous semigroup with generator $A$. Then the following hold:
\begin{abc}
\item $A$ is bi-closed, i.e., whenever $x_n\stackrel{\tau}{\to}x$ and $Ax_n\stackrel{\tau}{\to}y$ and both sequences are norm-bounded, then $y\in\dom(A)$ and $Ax=y$.
\item $\dom(A)$ is bi-dense in $X_0$, i.e., for each $x\in X_0$ there exists a norm-bounded sequence $(x_n)_{n\in\NN}$ in $\dom(A)$ such that $x_n\stackrel{\tau}{\to}x$.
\item For $x\in\dom(A)$ we have $T(t)x\in\dom(A)$ and $T(t)Ax=AT(t)x$ for all $t\geq0$.
\item For $t>0$ and $x\in X_0$ one has \begin{align}\int_0^t{T(s)x\ \dd s}\in\dom(A)\ \ \text{and}\ \ A\int_0^t{T(s)x\ \dd s}=T(t)x-x \end{align}
\item For $\lambda>\omega_0(T)$ one has $\lambda\in\rho(A)$ (thus $A$ is closed) and for $x\in X$ holds: \begin{align}\label{eq:bicontlaplace}
R(\lambda,A)x=\int_0^{\infty}{\ee^{-\lambda s}T(s)x\ \dd s}\end{align} where the integral is a $\tau$-improper integral.
\end{abc}
\end{theorem}

 Recall the following result of K\"uhnemund from \cite{Ku}, whose proof is originally based on integrated semigroups. We present here a short proof based on extrapolation spaces.

\begin{theorem}[K\"uhnemund]\label{thm:bistrong}
Let $(X_0,\|\cdot\|,\tau)$ be a triple satisfying Assumption \ref{asp:bicontspace}, and let $A$ be a linear operator on the Banach space $X_0$. The following are equivalent:
\begin{iiv}
\item $A$ is the generator of a bi-continuous semigroup $(T(t))_{t\geq0}$ of type $(M,\omega)$.
\item $A$ is a Hille--Yosida operator of type $(M,\omega)$,
i.e.,
\[
\|R(s,A)^{k}\|\leq\frac{M}{(s-\omega)^k}
\]
for all $k\in\NN$ and for all $s>\omega$.
$A$ is bi-densely defined and the family
\begin{equation}\label{eq:resbiequi}
\bigl\{(s-\alpha)^kR(s,A)^k:\ k\in\NN,\ s\geq\alpha\bigr\}
\end{equation}
 is bi-equicontinuous for each $\alpha>\omega$, meaning that for each norm bounded $\tau$-null sequence $(x_n)$ one has $(s-\alpha)^kR(s,A)^kx_n\to 0$ in $\tau$ uniformly for $k\in\NN$ and $s\geq\alpha$ as $n\to \infty$.
\end{iiv}
In this case we have the Euler formula
\[
T(t)x:=\tlim_{m\to\infty}\left(\frac{m}{t}R\left(\frac{m}{t},A\right)\right)^mx
\]
for each $x\in X_0$. Moreover, the subspace $\xX_0:=\overline{\dom(A)}\subseteq X_0$ is the space of norm strong continuity for $(T(t))_{t\geq0}$, it is invariant under the semigroup, and $(\tT(t))_{t\geq0}:=(T(t)|_{\xX_0})_{t\geq0}$ is the strongly continuous semigroup on $\xX_0$ generated by the part of $A$ in $\xX_0$.
\end{theorem}
\begin{proof}
That $\xX_0$ is the space of norm strong continuity for a bi-continuous semigroup $(T(t))_{t\geq 0}$ follows from Lemma \ref{lem:spacecont}. 

\medskip\noindent 
 We only prove the implication (ii) $\Rightarrow$ (i) and Euler formula; the other implication is standard. We may suppose that $\omega<0$. Since $A$ is a Hille--Yosida operator, the part $\aA_0$ of $A$ in $\xX_0$ generates a $C_0$-semigroup $(\tT(t))_{t\geq0}$ on the space $\xX_0:=\overline{\dom(A)}$. Define the function
\[
F(s):=\begin{cases}
 \frac1sR(\frac1s,A)&\text{for } s>0,\\
\Id&\text{for } s=0,
\end{cases}
\]
which is strongly continuous on $\xX_0$ by Remark \ref{rem:2}. Moreover, we have the Euler formula
\[
\tT_0(t)x=\lim_{m\to\infty}{F\bigl(\tfrac{t}{m}\bigr)^mx}
\]
 for $x\in \xX_0$ with convergence being uniform for $t$ in compact intervals $[0,t_0]$, see, e.g., \cite[Section~III.5(a)]{EN}. Since $R(\lambda,A)|_{\xX_0}=R(\lambda, \aA_0)$ and since $\dom(A)$ is bi-dense in $X_0$, by the local bi-equicontinuity assumption in \eqref{eq:resbiequi} we conclude that for $x\in X_0$ and $t>0$ the limit
\begin{equation}\label{eq:taueuler}
S(t)x:=\tlim_{m\to\infty}F\bigl(\tfrac{t}{m}\bigr)^mx
\end{equation}
exists, and the convergence is uniform for $t$ in compact intervals $[0,t_0]$. It follows that $t\mapsto S(t)x$ is $\tau$-strongly continuous for each $x\in X_0$. The operator family $(S(t))_{t\geq 0}$ is locally bi-equicontinuous because of the bi-equicontinuity assumption in \eqref{eq:resbiequi}.

\medskip\noindent Next, we prove that $\tT(t)$ leaves $\dom(A)$
invariant. Let $x\in \dom(A)$, so that $x=A^{-1}y$ for some $y\in X_0$, and insert $x$ in the previous formula \eqref{eq:taueuler} to obtain 
\begin{equation}\label{eq:invA}
\tT(t)x=S(t)A^{-1}y=\tlim_{m\to\infty}F\bigl(\tfrac{t}{m}\bigr)^m A^{-1}y=A^{-1}\tlim_{n\to\infty}F\bigl(\tfrac{t}{m}\bigr)^m y=A^{-1}S(t)y\in\dom(A),
\end{equation}
where we have used the bi-continuity of $A^{-1}$ and the boundedness of $(\bigl[\tfrac{n}{t}R\bigl(\tfrac{n}t,A\bigr)\bigr]^ny)_{n\in \NN}$. By Proposition \ref{prop:extT} (b) we can extend $\tT(t)$ to $X_0$ by setting $T(t):=A\tT(t)A^{-1}\in \LLL(X_0)$. It follows that $(T(t))_{t\geq 0}$ is a semigroup. By formula \eqref{eq:invA}, we have then $T(t)y=A\tT(t)A^{-1}y=AA^{-1}S(t)y=S(t)y$ for each $y\in X_0$. So that $(T(t))_{t\geq 0}$, coinciding with $(S(t))_{t\geq0}$, is locally bi-equicontinuous, and hence a bi-continuous semigroup.

\medskip It remains to show that the generator of $(T(t))_{t\geq 0}$ is precisely $A$. Let $B$ denote the generator of $(T(t))_{t\geq 0}$. Then, for large $\lambda>0$ and $x\in \xX_0$, we have
\[
R(\lambda,B)x=\int_0^\infty \ee^{-\lambda s}T(s)x\ \dd s=\int_0^\infty \ee^{-\lambda s}\tT(s)x\ \dd s=R(\lambda, \aA_0)x=R(\lambda,A)x.
\]
Since $R(\lambda,B)$ and $R(\lambda,A)$ are sequentially $\tau$-continuous on norm bounded sets and since $\dom(A)$ is bi-dense in $X_0$, we obtain $R(\lambda,B)=R(\lambda,A)$. This finishes the proof.
\end{proof}


The first statement in the next proposition is proved by Nagel and Sinestrari, see \cite{NagelSurvey} and \cite{nagel1993inhomogeneous}, while the second one follows directly from the results in Section \ref{sec:invert}.

\begin{proposition}
Let $A$ be Ha Hille--Yosida operator on the Banach space $X_0$ with domain $\dom(A)$. Denote by $(\tT(t)_{t\geq 0}$ the $C_0$-semigroup on $\xX_0=\overline{\dom(A)}$ generated by the part $\aA$ of $A$.
\begin{abc}
\item There is a one-parameter semigroup $(\TT(t))_{t\geq 0}$ on $\Fav_0(A)$ which extends $(\tT(t)_{t\geq 0}$. This semigroup is strongly continuous for the $\|\cdot\|_{-1}$ norm.
\item Suppose for each $t\geq 0$ the operator $\tT(t)$ leaves $\dom(A)$ invariant. The space $X_0$ is invariant under the semigroup operators $\TT(t)$ for every $t\geq 0$, i.e., $T(t)\in\LLL(X_0)$ for $T(t):=\TT(t)|_{X_0}$.
\end{abc}
\end{proposition}
%

\subsection{Extrapolated semigroups}
Next we extend a a bi-continuous semigroup on $X_0$ to the extrapolation space $X_{-1}$ as a bi-continuous semigroup. We have to handle two topologies, and the next proposition provides means to describe an additional locally convex topology on $X_{-1}$ still satisfying Assumption \ref{asp:bicontspace}. 
The definition of extrapolated spaces can be based on this proposition as was suggested by Haase \cite{Haase}.

\begin{proposition}\label{prop:Edef}
Let the triple $(X_0,\|\cdot\|,\tau)$ satisfy Assumption \ref{asp:bicontspace}, let $\semis$ be as in Remark \ref{rem:seminorm}, let $E$ be a vector space over $\CC$, and let $B:X_0\to E$ be a bijective linear mapping.
We define for $e\in E$ and $p\in \semis$
\[
\|e\|_E:=\|B^{-1}e\|,\quad\text{and}\quad p_{E}(e):=p(B^{-1}e).
\]
Then the following assertions hold:
\begin{abc}
\item $\|\cdot\|_E$ is a norm, $p_E$ is a seminorm for each $p\in \semis$.
\item With $\semis_E:=\{p_E:p\in \semis\}$ an with $\tau_E$ being the topology defined by $\semis_E$, the triple $(E,\|\cdot\|_E,\tau_E)$ satisfies the conditions in Assumption \ref{asp:bicontspace}.
\item If $(T(t))_{t\geq 0}$ is a bi-continuous semigroup on $X_0$ with respect to the topology $\tau$, then by $T_E(t):=BT(t)B^{-1}$ we define a bi-continuous semigroup on $E$. If $A$ is the generator of $(T(t))_{t\geq 0}$, then $BAB^{-1}$ is the generator of $(T_E(t))_{t\geq 0}$.
\end{abc}
\end{proposition}
\begin{proof}
Assertion (a) is evident. The conditions (1){} and (2){} from Assumption \ref{asp:bicontspace} are obviously satisfied by the definition of $\|\cdot\|_E$ and $p_E$. For condition (3) note that
\[
\|e\|_E=\|B^{-1}e\|=\sup_{p\in \semis}p(B^{-1}e)=\sup_{p_E\in \semis_E}p_E(e),
\]
so that, by Remark \ref{rem:pnorming}, (3){} in Assumption \ref{asp:bicontspace} is fulfilled. The proof of (b) is complete.

\medskip\noindent (c) For $e\in E$ we have $\|T_E(t)\|_E=\|B^{-1}BT(t)B^{-1}e\|=\|T(t)B^{-1}e\|\leq \|T(t)\|\cdot \|e\|_E$. Showing that $T_E(t)\in\LLL(E)$. That $(T_E(t))_{t\geq 0}$ satisfies the semigroup property is evident. For $e\in E$ and $p_E\in \semis_E$ we have
\[
p_E(T_E(t)e-e)=p(B^{-1}BT(t)B^{-1}e-B^{-1}e)=p(T(t)B^{-1}e-B^{-1}e)\to 0\quad \text{for $t\to 0$},
\]
showing the $\tau_E$-strong continuity of $(T_E(t))_{t\geq 0}$. If $(e_n)$ is a $\|\cdot\|_E$-bounded, $\tau_E$-null sequence, then $(B^{-1}e_n)$ is a $\|\cdot\|$-bounded $\tau$-null sequence, so that by assumption $T_E(t)e_n=T(t)B^{-1}e_n\to 0$ uniformly for $t$ in compact intervals.
If $A$ is the generator of $(T(t))_{t\geq 0}$, then by means of \eqref{eq:bicontlaplace} we can conclude that $B^{-1}AB$ is the generator of $(T_E(t))_{t\geq 0}$. 
\end{proof}

\begin{definition}\label{def:bicontext}
Let $(T(t))_{t\geq 0}$ be a bi-continuous semigroup in $X_0$ with generator $A$.
\begin{abc}
\item For $B=A^{-1}:X_0\to X_1$ and $E=X_1$ in Proposition \ref{prop:Edef} define $\semis_1:=\semis_E$, $\tau_{1}:=\tau_E$, $(T_1(t))_{t\geq 0}=(T_E(t))_{t\geq 0}$. 
\item For $B=A_{-1}:X_0\to X_{-1}$ and $E=X_{-1}$ in Proposition \ref{prop:Edef} define $\semis_{-1}:=\semis_E$, $\tau_{-1}:=\tau_E$, $(T_{-1}(t))_{t\geq 0}=(T_E(t))_{t\geq 0}$. 
\end{abc}
\end{definition}
We obtain immediately.

\begin{proposition}\label{prop:T-1bicont}
By construction $(T_1(t))_{t\geq 0}$ and $(T_{-1}(t))_{t\geq 0}$ are bi-continuous semigroups on $X_1$ and $X_{-1}$, with generators $A_1=A\rest\dom(A)$ and $A_{-1}$, respectively.
\end{proposition}
Iterating the procedure in Definition \ref{def:bicontext} we obtain the (extrapolated) semigroups $(T_n(t))_{t\geq 0}$ for the full scale $n\in \ZZ$.

\begin{definition}\label{def:bicontext2}
Let $(T(t))_{t\geq 0}$ be a bi-continuous semigroup on $X_0$ with generator $A$. Recursively we make the following definitions: 
\begin{abc}
\item For $B=A_n^{-1}:X_n\to X_{n+1}$ and $E=X_{n+1}$ in Proposition \ref{prop:Edef} define $\semis_{n+1}:=\semis_E$, $\tau_{n+1}:=\tau_E$, $(T_{n+1}(t))_{t\geq 0}=(T_n(t))_{t\geq 0}$. 
\item For $B=A_{-n-1}:X_{-n}\to X_{-n-1}$ and $E=X_{-n-1}$ in Proposition \ref{prop:Edef} define $\semis_{-n-1}:=\semis_E$, $\tau_{-n-1}:=\tau_E$, $(T_{-n-1}(t))_{t\geq 0}=(T_{-n}(t))_{t\geq 0}$. 
\end{abc}
\end{definition}

\begin{proposition}
For each $n\in \ZZ$ the semigroup $(T_n(t))_{t\geq 0}$ is bi-continuous on $X_n$ with generator $A_n:X_{n+1}\to X_n$. Its space of norm strong continuity is $\xX_n$.
\end{proposition}

\begin{proof}
The first statement follows directly from Proposition \ref{prop:T-1bicont} by recursion. For $n=0$ the second assertion is the content of Lemma \ref{lem:spacecont}, for general $n\in \ZZ$ one can argue recursively.
\end{proof}

The following diagram summarizes the situation:
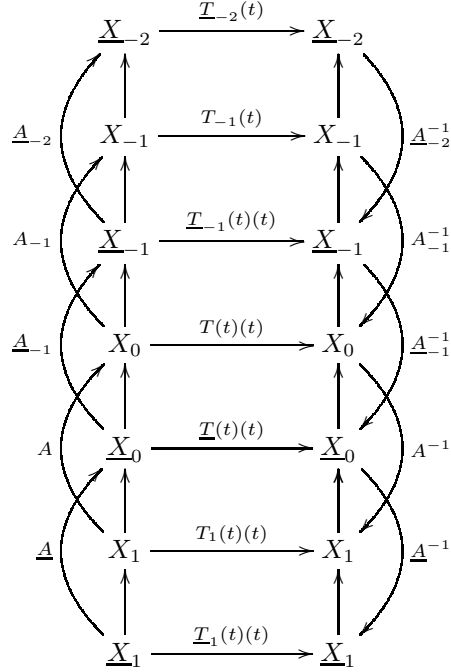
\begin{align*}
\xymatrix{ 
\xX_{-2}\ar[rr]^{\tT_{-2}(t)}& & \xX_{-2}\ar@/^2pc/[dd]^{\aA^{-1}_{-2}}\\
X_{-1}\ar[u]\ar[rr]^{T_{-1}(t)}& & X_{-1}\ar[u]\ar@/^2pc/[dd]^{A^{-1}_{-1}}\\
\xX_{-1}\ar[u]\ar[rr]^{\tT_{-1}(t)(t)}\ar@/^2pc/[uu]^{\aA_{-2}} & & \xX_{-1}\ar[u]\ar@/^2pc/[dd]^{\aA^{-1}_{-1}}\\
X_0\ar[u]\ar[rr]^{T(t)(t)}\ar@/^2pc/[uu]^{A_{-1}} & & X_0\ar[u]\ar@/^2pc/[dd]^{A^{-1}}\\
\xX_0\ar[u]\ar[rr]^{\tT(t)(t)}\ar@/^2pc/[uu]^{\aA_{-1}} & & \xX_0\ar[u]\ar@/^2pc/[dd]^{\aA^{-1}}\\
X_1\ar[rr]^{T_1(t)(t)}\ar@/^2pc/[uu]^{A}\ar[u] & & X_1\ar[u]\ar[u]\\
\xX_1\ar[u]\ar[rr]^{\tT_1(t)(t)}\ar@/^2pc/[uu]^{\aA} & & \xX_1\ar[u]\\
}
\end{align*}
The spaces $\xX_{n+1}$ are bi-dense in $X_n$ for the topology $\tau_{X_n}$ and dense in $\xX_n$ for the norm $\|\cdot\|_{X_n}$. The semigroups $(T_n(t))_{t\geq 0}$ are bi-continuous on $X_n$, while $(\tT_n(t))_{t\geq 0}$ are $C_0$-semigroups (strongly continuous for the norm) on $\xX_n$.

\subsection{H\"older spaces of bi-continuous semigroups}
Suppose $A$ generates the bi-continuous semigroup $(T(t))_{t\geq0}$ of negative growth bound on $X_0$. Recall from Theorem \ref{thm:bistrong} that the restricted operators $\tT(t):=T(t)|_{\xX_0}$ form a $C_0$-semigroup $(\tT(t))_{t\geq0}$ on $\xX_0$. Also recall from Proposition \ref{prop:FavCont} that for $\alpha\in(0,1]$
\[
\Fav_{\alpha}(A)=\Bigl\{x\in\xX_0:\ \sup_{t>0}{\frac{\|\tT(t)x-x\|}{t^{\alpha}}}<\infty\Bigr\}=\Bigl\{x\in X_0:\ \sup_{t>0}{\frac{\|T(t)x-x\|}{t^{\alpha}}}<\infty\Bigr\}\]
with the norm
\[
\|x\|_{\alpha}=\sup_{t>0}{\frac{\|\tT(t)x-x\|}{t^{\alpha}}},
\]
and for $\alpha\in(0,1)$:
\[
\XHol_{\alpha}(A):=\Bigl\{x\in\xX_0:\ \lim_{t\to0}{\frac{\|\tT(t)x-x\|}{t^{\alpha}}}=0\Bigr\}=\Bigl\{x\in X_0:\ \lim_{t\to0}{\frac{\|T(t)x-x\|}{t^{\alpha}}}=0\Bigr\}.
\]
 We have the following (continuous) inclusions:
\[
\xX_1\hookrightarrow X_1\to \XHol_{\alpha}(A)\hookrightarrow\Fav_{\alpha}(A)\to\xX_0\hookrightarrow X_0;
\]
all these spaces are invariant under $(T(t))_{t\geq0}$.We now extend this diagram by a space which lies between $\xX_{\alpha}$ and $F_{\alpha}$. 
\begin{definition}\label{def:biHoelder}
Let $(T(t))_{t\geq0}$ be a bi-continuous semigroup of negative growth bound on a Banach space $X_0$ with respect to a locally convex topology $\tau$ that is generated by a family $\semis$ of seminorms satisfying \eqref{eq:semisnorm}. For $\alpha\in(0,1)$ we define the space 
\begin{align}
&X_{\alpha}:=\Bigl\{x\in X_0:\ \tlim_{t\to0}{\frac{T(t)x-x}{t^{\alpha}}}=0\ \text{and}\ \sup_{t>0}{\frac{\|T(t)x-x\|}{t^{\alpha}}}<\infty\Bigr\},
\end{align}
and endow it with the norm $\|\cdot\|_{F_{\alpha}}$,
We further equip $F_{\alpha}$ and $X_\alpha$ with the locally convex topology $\tau_{F_{\alpha}}$ generated by the family of seminorms $\semis_{F_{\alpha}}:=\{p_{F_{\alpha}}\}$, where $p_{F\alpha}$ is given by
\begin{align}
p_{F_{\alpha}}(x):=\sup_{t>0}{\frac{p(T(t)x-x)}{t^{\alpha}}}.
\end{align}
\end{definition}

It is easy to see that $X_\alpha$ is a Banach space, i.e., as closed subspace of $F_\alpha$.
By construction we have that $\XHol_{\alpha}(A)\subseteq X_{\alpha}\subseteq\Fav_{\alpha}(A)$, which was our first goal. 
Next we discuss some properties of this space.
\begin{lemma}
\begin{abc}
\item Let $(x_n)$ be a $\|\cdot\|_{F_\alpha}$-norm bounded sequence in $F_\alpha$ with $x_n\to x\in X_0$ in the topology $\tau$. Then $x\in F_\alpha$.

\item The triple $(F_{\alpha},\|\cdot\|_{F_\alpha},\tau_{F_\alpha})$ satisfies the conditions in Assumption \ref{asp:bicontspace}.
\item
$X_{\alpha}$ is bi-closed in $F_{\alpha}$, i.e., every $\|\cdot\|_{F_\alpha}$-bounded an $\tau_{F_\alpha}$-convergent sequence in $X_{\alpha}$ has its limit in $X_{\alpha}$.
\end{abc}
\end{lemma}

\begin{proof}
\medskip\noindent (a) The statement follows from the fact that the norm $\|\cdot\|_{F_\alpha}$ is lower semicontinuous for the topology $\tau$. In detail:
Let $M\geq 0$ be such that \[\frac{\|T(t)x_n-x_n\|}{t^{\alpha}}\leq \|x_n\|_{F_\alpha}\leq M\] for each $n\in\NN$ and $t>0$. Hence we can estimate as follows:
\begin{align*}
\sup_{t>0}\frac{\|T(t)x-x\|}{t^{\alpha}}&=\sup_{t>0}\sup_{p\in\semis}p\Bigl(\frac{T(t)x-x}{t^{\alpha}}\Bigr)=\sup_{t>0}\sup_{p\in\semis}\lim_{n\to\infty}p\Bigl(\frac{T(t)x_n-x_n}{t^{\alpha}}\Bigr)\\
&\leq \sup_{t>0}\sup_{p\in\semis}\limsup_{n\to\infty}\Bigl\|\frac{T(t)x_n-x_n}{t^{\alpha}}\Bigr\|\leq \sup_{t>0}\sup_{n\in\NN}\Bigl\|\frac{T(t)x_n-x_n}{t^{\alpha}}\Bigr\|\leq M.
\end{align*}

\medskip\noindent (b) We have for $p\in\semis$ and $x\in F_\alpha$ that
\[
p_{F_{\alpha}}(x)=\sup_{t>0}{\frac{p(T(t)x-x)}{t^{\alpha}}}\leq\sup_{t>0}{\frac{\|T(t)x-x\|}{t^{\alpha}}}=\|x\|_{F_{\alpha}}.
\]
This proves that $\tau_{F_{\alpha}}$ is weaker than the $\|\cdot\|_{F_{\alpha}}$-topology, but $\tau_{F_\alpha}$ is still Hausdorff by construction. For the second property of Assumption \ref{asp:bicontspace} let $(x_n)_{n\in\NN}$ be a $\tau_{F_{\alpha}}$-Cauchy sequence in $F_{\alpha}$ such that there exists $M>0$ with $\|x_n\|_{F_{\alpha}}\leq M$ for each $n\in\NN$. Since $\tau$ is weaker than $\tau_{F_\alpha}$, we conclude that $(x_n)$ is $\tau$-Cauchy sequence which is also bounded in $\|\cdot\|_{F_\alpha}$ hence in $\|\cdot\|$. By assumption there is $x\in X_0$ such that $x_n\to x$ in $\tau$. 
By part (a) we obtain $x\in F_\alpha$. It remains to prove that $x_n\to x$ in $\tau_{F_\alpha}$. Let $\veps>0$, and take $N\in\NN$ such that for each $n,m\in \NN$ with $n,m\geq N$ we have $p_{F_\alpha}(x_n-x_m)<\veps$. For $t>0$ 
\[
p\Bigl(\frac{T(t)(x_n-x)-(x_n-x)}{t^{\alpha}}\Bigr)=\lim_{m\to\infty}
p\Bigl(\frac{T(t)(x_n-x_m)-(x_n-x_m)}{t^{\alpha}}\Bigr)\leq p_{F_\alpha}(x_n-x_m)<\veps
\]
for each $n\geq N$. Taking supremum in $t>0$ we obtain $p_{F_\alpha}(x-x_n)\leq \veps$ for each $n\geq N$.

\medskip\noindent The norming property in \eqref{eq:norm} follows again simply by applying the argumentation from Remark \ref{rem:seminorm} and the fact that the family $\semis$ is already norming by assumption.

\medskip\noindent (c) Let $(x_n)_{n\in\mathbb{N}}$ be a $\|\cdot\|_{F_{\alpha}}$-bounded and $\tau_{F_{\alpha}}$ convergent sequence in $X_{\alpha}$ with limit $x\in X_0$. For $p\in \semis$ we then have 
\[
\sup_{t>0}{p\Bigl(\frac{T(t)(x_n-x)-(x_n-x)}{t^{\alpha}}\Bigr)}\to 0.
\]
Since $x_n\in X_{\alpha}$ for each $n\in\mathbb{N}$, we have
\[
\lim_{t\to0}{p\Bigl(\frac{T(t)x_n-x_n}{t^{\alpha}}\Bigr)}=0, \quad\text{and}\quad \sup_{t>0}{\Bigl\|\frac{T(t)x_n-x_n}{t^{\alpha}}\Bigr\|}<\infty.
\]
We now can conclude for a fixed $p\in \semis$
\begin{align*}
p\Bigl(\frac{T(t)x-x}{t^{\alpha}}\Bigr)&=p\Bigl(\frac{T(t)(x-x_n)-(x-x_n)+T(t)x_n-x_n}{t^{\alpha}}\Bigr)\\
&\leq p\Bigl(\frac{T(t)(x-x_n)-(x-x_n)}{t^{\alpha}}\Bigr)+p(\frac{T(t)x_n-x_n}{t^{\alpha}}\Bigr)\\
&\leq p_{F_{\alpha}}(x-x_n)+p\Bigl(\frac{T(t)x_n-x_n}{t^{\alpha}}\Bigr)<\frac{\varepsilon}2+\frac{\varepsilon}2=\varepsilon,
\end{align*}
where we first fix $n\in\NN$ such that $p_{F_{\alpha}}(x-x_n)<\frac{\varepsilon}{2}$, and then we take $\delta>0$ such that $0<t<\delta$ implies $p(\frac{T(t)x_n-x_n}{t^{\alpha}})<\frac{\varepsilon}{2}$.
\end{proof}

The next goal is to see that $(T(t))_{t\geq 0}$ can be restricted to $X_{\alpha}$ to obtain a bi-continuous semigroup with respect to the topology $\tau_{F_\alpha}$. 

\begin{remark}
For the proof of the next lemma we notice that we have an equivalent formulation of the fact that $$\displaystyle{{\tlim}_{t\to0}\frac{T(s)x-x}{s^{\alpha}}=0}$$
 by means of sequences in $\RR$. In fact, to prove this convergence to $0$ we only have to check that $$\frac{p(T(s_n)x-x)}{s_n^{\alpha}}\to0$$ for $n\to\infty$ for every null-sequence $(s_n)_{n\in\NN}$ in $[0,\infty)$ and for each $p\in \semis$. 
\end{remark}

\begin{lemma}
If $(T(t))_{t\geq0}$ is a bi-continuous semigroup, then $X_{\alpha}$ is invariant under the semigroup.
\end{lemma}
\noindent 
 
\begin{proof}
 Let $x\in X_{\alpha}$. Then we have that $y_n:=\frac{T(s_n)x-x}{s_n^{\alpha}}$ converges to $0$ with respect to $\tau$ if $(s_n)_{n\in\mathbb{N}}$ is any null-sequence and $n\to\infty$. Moreover, this sequence $(y_n)_{n\in\mathbb{N}}$ is $\|\cdot\|$-bounded by the assumption that $x\in X_{\alpha}$. Whence we conclude
 \[
 \tlim_{n\to\infty}T(t)y_n=\tlim_{n\to\infty}\frac{T(s_n)T(t)x-T(t)x}{s_n^{\alpha}}=0,
 \]
 so that $T(t)x\in X_{\alpha}$.
\end{proof}

To prove that $(T(t))_{t\geq0}$ is bi-continuous on $X_{\alpha}$ we have to show that the semigroup satisfies all conditions from Definition \ref{def:bicontsemi}. Notice that the local boundedness and the semigroup property are trivial.

\begin{proposition}
\label{prop:strongcontalpha}
If $(T(t))_{t\geq0}$ is a bi-continuous semigroup on $X_0$ and $\alpha\in(0,1)$, then $(T(t))_{t\geq0}$ is strongly $\tau_{F_{\alpha}}$-continuous on $X_{\alpha}$.
\end{proposition}

\begin{proof}We have to show that $p_{F_{\alpha}}(T(t_n)x-x)\to 0$ for all $p\in \semis$ whenever $t_n\downto 0$. Let $s_n,t_n>0$ be with $s_n,t_n\to 0$. Then
\begin{align}
\notag\frac{p(T(s_n)T(t_n)x-T(s_n)x-T(t_n)x+x)}{s_n^{\alpha}}&\leq\frac{p(T(t_n)T(s_s)x-T(t_n)x)}{s_n^{\alpha}}+\frac{p(T(s_n)x-x)}{s_n^{\alpha}}\\
&\label{eq:last} =\frac{p(T(t_n)(T(s_n)x-x))}{s_n^{\alpha}}+\frac{p(T(s_n)x-x)}{s_n^{\alpha}}.
\end{align}
 The sequence $(y_n)$ given by $y_n:=\frac{T(s_n)x-x}{s_n^{\alpha}}$ is $\|\cdot\|$-bounded and $\tau$-convergent to $0$, because $x\in X_\alpha$. So that the last term in the previous equation \eqref{eq:last} converges to $0$. But since $\{T(t_n):n\in \NN\}$ is bi-equicontinuous, also the first term in \eqref{eq:last} converges to $0$. This proves strong continuity with respect to $\tau_{F_{\alpha}}$.
\end{proof}

To conclude with the result that $(T(t))_{t\geq0}$ is bi-continuous on $X_{\alpha}$ we have to show that this restriction is in particular locally bi-equicontinuous. 

\begin{proposition}Let $(T(t))_{t\geq0}$ be a bi-continuous semigroup on $X_0$. Then $(T(t))_{t\geq0}$ is locally bi-equicontinuous on $F_{\alpha}$.\end{proposition}

\begin{proof}Let $(x_n)_{n\in\NN}$ be a $\|\cdot\|_{F_{\alpha}}$-bounded sequence which converges to zero with respect to $\tau_{F_{\alpha}}$ and assume that $(T(t)x_n)_{n\in\NN}$ does not converges to zero uniformly for $t\in[0,t_0]$ for some $t_0>0$. Hence there exists $p\in\semis$, $\delta>0$ and a sequence $(t_n)_{n\in\NN}$ of positive real numbers such that 
\[
p_{F_{\alpha}}(T(t_n)x_n)>\delta
\]
for all $n\in\NN$. As a consequence there exists a sequence $(s_n)_{n\in\NN}$ in $\RR$ which is a null-sequence such that
\[
\frac{p(T(s_n)T(t_n)x_n-T(t_n)x_n)}{s_n^{\alpha}}>\delta
\]
for each $n\in\NN$. Now notice that the sequence $(y_n)_{n\in\NN}$ defined by $y_n:=\frac{T(s_n)x_n-x_n}{s_n^{\alpha}}$ is a $\tau$-null sequence since for $q\in\semis$:
\[
\frac{q(T(s_n)x_n-x_n)}{s_n^{\alpha}}\leq\sup_{s>0}\frac{q(T(s)x_n-x_n)}{s^{\alpha}}
\]
and the term on the right hand side converges to zero as $n\to\infty$ by assumption. Using the local bi-equicontinuity of the semigroup $(T(t))_{t\geq0}$ with respect to $\tau$, we conclude that $\frac{T(t)T(s_n)x_n-T(t)x_n}{s_n}$ converges to zero uniformly for $t\in[0,t_0]$, contradiction. Hence we conclude that $(T(t))_{t\geq0}$ is locally bi-equicontinuous on $X_{\alpha}$.
\end{proof}

\begin{remark}
Notice that the local bi-equicontinuity with respect to $\tau_{F_\alpha}$ holds on the whole space $F_{\alpha}$, while strong $\tau_{F_\alpha}$-continuity holds on $X_{\alpha}$ only. In particular, we will see in Theorem \ref{thm:strong}, that $X_{\alpha}$ is the space of strong $\tau_{F_{\alpha}}$-continuity.
\end{remark}

We can summarize the previous results in the following theorem.

\begin{theorem}
Let $(T(t))_{t\geq0}$ be a bi-continuous semigroup on $X_0$. Then the restricted operators $T_\alpha(t):=T(t)|_{X_\alpha}$ to $X_{\alpha}$ form a bi-continuous semigroup. Moreover, the generator $A_\alpha$ of $(T_\alpha(t))_{t\geq 0}$ is the part of $A$ in $X_{\alpha}$ which is continuous with respect to $\tau_{F_{\alpha}}$ and $\tau_{F_{\alpha-1}}$.
\end{theorem}

\begin{proof}
We only have to prove the last part of the theorem, namely that the part of $A$ in $X_{\alpha}$ generates the restricted semigroup on $X_{\alpha}$. The proof goes similarly to that of Proposition in \cite[Chap. II, Par. 2.3]{EN}. Since the embedding $X_{\alpha}\subseteq X_0$ is continuous for the topologies $\tau_{F_\alpha}$ and $\tau$, we conclude that $A_\alpha\subseteq A\rest_{X_{\alpha}}$. For the converse take $\lambda\in\RR$ large enough such that
\[
R(\lambda,C)x=\int_0^{\infty}{\ee^{-\lambda s}T(s)x\ \dd s}=R(\lambda,A)x,\ \ x\in X_{\alpha}.
\]
For $x\in\dom(A_{|X_{\alpha}})$ we obtain
\[
x=R(\lambda,A)(\lambda-A)x=R(\lambda,C)(\lambda-A)x\in\dom(C)
\]
and hence $A_{|X_{\alpha}}\subseteq A_\alpha$. This proves that the part of $A$ in $X_{\alpha}$ generates the restricted semigroup.

\end{proof}

By similar reasoning as in Lemma \ref{lem:spacecont} one can prove the following result:
\begin{theorem}\label{thm:strong}
Let $\alpha\in(0,1)$ and let $(T(t))_{t\geq0}$ be a bi-continuous semigroup on $X$. Then $\dom(A)$ is $\tau_{F_{\alpha}}$-bi-dense in $X_{\alpha}$ and
\begin{equation}\label{eq:xalphacont}
X_{\alpha}=\bigl\{x\in F_{\alpha}:\ \tau_{F_{\alpha}}\lim_{t\to0}T(t)x=x\bigr\},
\end{equation}
i.e., for $x\in F_\alpha$ the mapping $t\mapsto T(t)x$ is $\tau_{F_\alpha}$-continuous, if and only if $x\in X_\alpha$.
\end{theorem}

\begin{proof}
Denote by $X_{\alpha,\conti}$ the right-hand side of \eqref{eq:xalphacont}, i.e., the space of $\tau_{F_\alpha}$-strong continuity.
Notice first of all that $\dom(A)\subseteq \xX_\alpha\subseteq X_\alpha\subseteq X_{\alpha,\conti}$.

\medskip\noindent Suppose $x\in X_{\alpha,\conti}$. For each $n\in \NN$ we have
\begin{align*}x_n:={n\int_0^\frac{1}{n}{T_\alpha(t)x\ \dd t}}={n\int_0^\frac{1}{n}{T(t)x\ \dd t}}\in\dom(A)\end{align*}
as a $\tau$- and $\tau_{F_\alpha}$-convergent Riemann integral. Whence it follows that $x_n\stackrel{\tau_{F_{\alpha}}}\to x$, whereas the $\|\cdot\|_{F_{\alpha}}$-boundedness of $(x_n)_{n\in\NN}$ clear. We conclude that $x\in X_\alpha$ (because $X_\alpha$ is bi-closed in $F_\alpha$), implying $X_{\alpha,\conti}\subseteq X_\alpha$. As a byproduct we also obtain that $\dom(A)$ is bi-dense in $X_\alpha$.
\end{proof}

\begin{proposition}
For $0\leq \alpha< \beta\leq 1$ we have
\[
X_1=\dom(A)\inj F_\beta\inj \xX_\alpha\subseteq X_\alpha,
\]
where the embeddings are continuous for the respective norms and for the respective topologies $\tau_{1}$, $\tau_{F_\beta}$, $\tau_{F_\alpha}$. 
The space $\dom(A)$ bi-dense in $X_\alpha$, and as a consequence $X_\beta$ is bi-dense in $X_\alpha$.
\end{proposition}

\subsection{Representation of H\"older spaces by generators}
Analogously to the Proposition \ref{prop:FavCont} we have a representation of the H\"older space $X_{\alpha}$ by means of the semigroup generator.

\begin{theorem}
Let $(T(t))_{t\geq0}$ be a bi-continuous semigroup with negative growth bound and generator $A$. 
For $\alpha\in(0,1)$ we have 
\begin{equation}
\label{eq:Xalphadef}
X_{\alpha}=\Bigl\{x\in X_0:\ \tlim_{\lambda\to\infty}{\lambda^{\alpha}AR(\lambda,A)x}=0\text{ and } \sup_{\lambda>0}{\|\lambda^{\alpha}AR(\lambda,A)x\|}<\infty\Bigr\}.\end{equation}
\end{theorem}

\begin{proof}
Suppose $x\in X_{\alpha}$. From Proposition \ref{prop:FavCont} we deduce immediately 
\[
\sup_{\lambda>0}{\|\lambda^{\alpha}AR(\lambda,A)x\|}<\infty.
\]
Let now $\varepsilon>0$ be arbitrary. Then, since $x\in X_{\alpha}$, for $p\in\semis$ we can find $\delta>0$ such that $0\leq t<\delta$ implies that $\frac{p(T(t)x-x)}{t^{\alpha}}<\varepsilon$. Recall the following formula for the resolvent:
\begin{align*}\lambda^{\alpha}AR(\lambda,A)x=\lambda^{\alpha+1}\int_0^{\infty}{\ee^{-\lambda s}(T(s)x-x)\ \dd s}.\end{align*} 
From this we deduce
\begin{align*}
p(\lambda^{\alpha}AR(\lambda,A)x)&\leq\lambda^{\alpha+1}\int_0^{\infty}{\ee^{-\lambda s}\cdot\frac{p(T(s)x-x)}{s^{\alpha}}s^{\alpha}\ \dd s}\\
&=\lambda^{\alpha+1}\int_0^{\delta}{\ee^{-\lambda s}\cdot\frac{p(T(s)x-x)}{s^{\alpha}}s^{\alpha}\ \dd s}+\lambda^{\alpha+1}\int_{\delta}^{\infty}{\ee^{-\lambda s}\cdot\frac{p(T(s)x-x)}{s^{\alpha}}s^{\alpha}\ \dd s}\\
&<\lambda^{\alpha+1}\varepsilon\int_0^{\delta}{\ee^{-\lambda s}s^{\alpha}\ \dd s}+\lambda^{\alpha+1}\int_{\delta}^{\infty}{\ee^{-\lambda s}\cdot\frac{\|T(s)x-x\|}{s^{\alpha}}s^{\alpha}\ \dd s}\\
&\leq\lambda^{\alpha+1}\varepsilon\int_0^{\delta}{\ee^{-\lambda s}s^{\alpha}\ \dd s}+\|x\|_{F_\alpha}\lambda^{\alpha+1}\int_{\delta}^{\infty}{\ee^{-\lambda s}\cdot s^{\alpha}\ \dd s}\\
&=\varepsilon\int_{0}^{\lambda\delta}{\ee^{-t}t^{\alpha}\ \dd t}+\|x\|_{F_\alpha}\int_{\lambda\delta}^{\infty}{\ee^{-t}t^{\alpha}\ \dd t}\\
&\leq L\varepsilon+\|x\|_{F_\alpha}\int_{\lambda\delta}^{\infty}{\ee^{-t}t^{\alpha}\ \dd t}
\end{align*}
where $L:=\int_0^{\infty}{\ee^{-\lambda s}s^{\alpha}\ \dd s}<\infty$. Notice that the last part of the sum tends to zero if $\lambda\to\infty$ since we fixed $\delta>0$ in the beginning. So that $\tlim_{\lambda\to\infty} \lambda^{\alpha}AR(\lambda,A)x=0$.

\medskip\noindent For the converse inclusion suppose that $\tlim_{\lambda\to\infty}{\lambda^{\alpha}AR(\lambda,A)x}=0$ and $\sup_{\lambda>0}{\|\lambda^{\alpha}AR(\lambda,A)x\|}<\infty$, the latter immediately implying $\|x\|_{F_\alpha}<\infty$ (see Proposition 
\ref{prop:FavCont}). We have to show now that $\tlim_{t\to0}{\frac{T(t)x-x}{t^{\alpha}}}=0$. For $\lambda>0$ define $x_{\lambda}=\lambda R(\lambda,A)$ and $y_{\lambda}=AR(\lambda,A)$, then we have $$x=\lambda R(\lambda,A)x-AR(\lambda,A)x=x_{\lambda}-y_{\lambda}.$$
Let $\varepsilon>0$ be arbitrary. First notice that for $p\in\semis$
\begin{equation}\label{eq:xl}
\frac{p(T(t)x_{\lambda}-x_{\lambda})}{t^{\alpha}}\leq\frac{1}{t^{\alpha}}p(T(t)\lambda R(\lambda,A)x-\lambda R(\lambda,A)x)\leq\frac{\lambda^{1-\alpha}}{t^{\alpha}}\int_0^t{p(T(s)\lambda^{\alpha}AR(\lambda,A)x)\ \dd s}.
\end{equation}
By assumption the term $\lambda^{\alpha}AR(\lambda,A)x$ is norm-bounded and converges in the topology $\tau$ to zero as $\lambda\to\infty$, hence by the local bi-equicontinuity we conclude that $p(T(s)\lambda^{\alpha}AR(\lambda,A)x)\to0$ uniformly for $s\in [0,1]$. Now let $\lambda_0>1$ so large that for $\lambda>\lambda_0$ and $s\in [0,1]$ we have $p(T(s)\lambda^{\alpha}AR(\lambda,A)x)<\varepsilon$. If $t<\frac{1}{\lambda_0}$, then $\lambda:=\frac{1}{t}>\lambda_0$ and we obtain, that the expression in \eqref{eq:xl} becomes smaller than $\varepsilon$.

For the estimate of the second part we observe:
\begin{equation*}
\frac{p(T(t)y_{\lambda}-y_{\lambda})}{t^{\alpha}}\leq\frac{1}{(t\lambda)^{\alpha}}p(T(t)\lambda^{\alpha}AR(\lambda,A)x)+\frac{1}{(t\lambda)^{\alpha}}p(\lambda^{\alpha}AR(\lambda,A)x).
\end{equation*}
Now by taking $t<\frac1{\lambda_0}$ and $\lambda:=\frac{1}{t}$ we obtain the estimate:
\begin{equation}\label{eq:yl}
\frac{p(T(t)y_{\lambda}-y_{\lambda})}{t^{\alpha}}\leq p(T(\tfrac{1}{\lambda})\lambda^{\alpha}AR(\lambda,A)x)+p(\lambda^{\alpha}AR(\lambda,A)x)<\veps+\veps,
\end{equation}
by the choice of $\lambda_0$.
Altogether we obtain for $t<\frac1{\lambda_0}$ that $
\frac{p(T(t)x-x)}{t^{\alpha}}<3\veps$, showing
\[
\tlim_{t\to0}\frac{T(t)x-x}{t^{\alpha}}=0,
\]
i.e., $x\in X_\alpha$ as required.
\end{proof}

\begin{remark}
We remark that it is possible to define the space $X_{\alpha}(A)$ as the right-hand side of \eqref{eq:Xalphadef} for operators $A$ which are not necessarily generators of bi-continuous semigroups but whose resolvent fulfills certain continuity assumptions with respect to a topology satisfying, say, Assumption \ref{asp:bicontspace}.
\end{remark}

Again, we put our spaces $X_{\alpha}$ in the general context of Theorem \ref{thm:iden}. 

\begin{proposition}\label{cor:ExtFav2}
For $\alpha\in\left(0,1\right)$ and $\mathcal{A}$, $\lambda$ and $\D$ as in Theorem \ref{thm:iden} we have
\begin{align*}
X_{-\alpha}&=\Bigl\{(\lambda-\mathcal{A})y\in X_{-1}:\ \sup_{t>0}{\frac{\left\|T(t)y-y\right\|}{t^{1-\alpha}}}<\infty,\ \tlim_{t\to0}{\frac{T(t)y-y}{t^{1-\alpha}}}=0\Bigl\}.
\end{align*}
\end{proposition}

Finally, we extend the scale of spaces $X_{\alpha}$  to the whole range $\alpha\in\RR$.

\begin{definition}
For $\alpha\in\RR\setminus\ZZ$ we write $\alpha=m+\beta$ with $m\in\ZZ$ and $\beta\in (0,1]$, and define
\[
X_{\alpha}(A):=X_{\beta}(A_m),
\]
also with the corresponding norms. The locally convex topology on $X_{\alpha}$ comes from $X_{\beta}$ via the mapping $A_m$.
\end{definition}

\begin{remark}
We summarize all previous results in the following diagram:
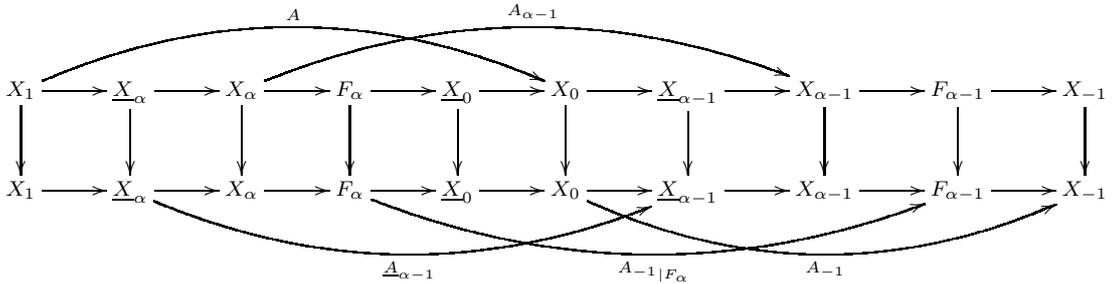
\footnotesize{
\begin{align*}
\xymatrix{ 
X_1\ar@/^2pc/[rrrrr]^{A}\ar[r]\ar[d]&\xX_{\alpha}\ar[r]\ar[d]&X_{\alpha}\ar[r]\ar[d]\ar@/^2pc/[rrrrr]^{A_{\alpha-1}}&F_{\alpha}\ar[r]\ar[d]&\xX_0\ar[r]\ar[d]&X_0\ar[r]\ar[d]&\xX_{\alpha-1}\ar[r]\ar[d]&X_{\alpha-1}\ar[r]\ar[d]&F_{\alpha-1}\ar[r]\ar[d]&X_{-1}\ar[d]\\
X_1\ar[r]&\xX_{\alpha}\ar[r]\ar@/_2pc/[rrrrr]_{\aA_{\alpha-1}}&X_{\alpha}\ar[r]&F_{\alpha}\ar@/_2pc/[rrrrr]_{{A_{-1}}_{|F_{\alpha}}}\ar[r]&\xX_0\ar[r]&X_0\ar[r]\ar@/_2pc/[rrrr]_{A_{-1}}&\xX_{\alpha-1}\ar[r]&X_{\alpha-1}\ar[r]&F_{\alpha-1}\ar[r]&X_{-1}
}
\end{align*}
}\\\normalsize
where $\alpha\in\left(0,1\right)$. Here $A_{\alpha-1}$ and $\aA_{\alpha-1}$ are defined to be the part of $A_{-1}$ in $X_{\alpha-1}$ and the part of $\aA_{-1}$ in $\xX_{\alpha-1}$, respectively and are all continuous with respect to the norms and topologies on the spaces. In addition we recall that $X_{\alpha-1}$ and $\xX_{\alpha-1}$ are the extrapolation spaces of $X_{\alpha}(A_{-1})$ and $\xX_{\alpha}(A_{-1})$, respectively. This shows that we can extend the space $X_{\alpha}$ from $\alpha\in\left(0,1\right)$ to $\alpha\in\RR$ by extra- and interpolation. All horizontal arrows are inclusions which are all continuous, whereas the vertical arrows are the action(s) of the semigroup(s). All spaces are dense in the underlined ones that contain them, while the spaces without underlining are bi-dense in each of the bigger ones. 
\end{remark}

\section{Examples}\label{sec:examp}
In this section we present examples for extrapolation and intermediate spaces for (generators of) bi-continuous semigroups. 
We will use Theorem \ref{thm:iden} and its variants to identify the space $X_{\alpha}$ with $\alpha<0$.

\subsection{The translation semigroup}
Let $X_0=\BC(\RR)$, the space of bounded and continuous functions equipped with the supremum norm $\|\cdot\|_{\infty}$ and consider thereon the compact-open topology $\tau_{\mathrm{co}}$ generated by the family of seminorms $\semis=\{p_K:\ K\subseteq\RR\ \text{compact}\}$, where 
\[
p_K(f)=\sup_{x\in K}{|f(x)|},\quad f\in\BC(\RR).
\]
The left translation semigroup $(T(t))_{t\geq0}$ defined by
\[
T(t)f(x)=f(x+t),\quad t\geq0
\]
is bi-continuous on $X_0$ with respect to $\tau_{co}$. The generator $A$ of this semigroup is given by the first derivative $Af=f'$ on the domain (see \cite{Ku})
\[
\dom(A)=\{f\in\BC(\RR):\ f\quad \text{is differentiable}\quad f'\in\BC(\RR)\}.
\]
The space of strong continuity is $\xX_0=\BUC(\RR)$, the space of all bounded, uniformly continuous functions. We use Theorem \ref{thm:iden} to determine the corresponding extrapolation spaces. To this purpose let $\D=\Distr(\RR)$ be the space of all distributions on $\RR$, let $\A:\D'(\RR)\to\D'(\RR)$ be the distributional derivative and let $i:\BC(\RR)\to\Distr(\RR)$ be the regular embedding. From Theorem \ref{thm:iden} it then follows 
\begin{align*}
X_{-1}&=\{F\in\Distr(\RR):\ F=f-Df\ \text{for some}\ f\in \BUC(\RR)\},\\
X_{-1}&=\{F\in\Distr(\RR):\ F=f-Df\ \text{for some}\ f\in \BC(\RR)\}.
\end{align*}
For the Favard and H\"older spaces we have
\begin{align*}
F_{\alpha}&=\Bigl\{f\in\BC(\RR):\ \sup_{\substack{x,y\in\RR\\x\neq y}}{\frac{|f(x)-f(y)|}{|x-y|^{\alpha}}}<\infty\Bigl\}=\BC^{\alpha}(\RR),\\
\xX_{\alpha}&=\Bigl\{f\in\BUC(\RR):\ \lim_{t\to0}{\sup_{\substack{x,y\in\RR\\0<|x-y|<t}}{\frac{|f(x)-f(y)|}{|x-y|^{\alpha}}}}=0\Bigl\}=\h_b^{\alpha}(\RR).
\end{align*}
Hence $F_{\alpha}$ can be identified with the space of bounded $\alpha$-H\"older-continuous functions and $\xX_{\alpha}$ with the so-called little H\"older space $\h_b^{\alpha}(\RR)$. For the abstract H\"older space $X_{\alpha}$ corresponding to the bi-continuous semigroup we obtain the local version $\h_{b,\text{loc}}^{\alpha}(\RR)$ of the little H\"older space:
\begin{align*}
\h_{b,\text{loc}}^{\alpha}&=\Bigl\{f\in\BC^{\alpha}(\RR):\ \lim_{t\to0}{\sup_{\substack{x,y\in K\\0<|x-y|<t}}{\frac{|f(x)-f(y)|}{|x-y|^{\alpha}}}}=0\ \text{for each}\ K\subseteq\RR\ \text{compact} \Bigr\}.
\end{align*}
Then $X_{\alpha}=h_{b,\text{loc}}^{\alpha}(\RR)$. It is easy to see $\xX_{\alpha}\subsetneq X_{\alpha}\subsetneq F_{\alpha}$. The extrapolated Favard class $F_0$ can be identified with $\Ell^{\infty}(\RR)$. We know from the general theory that $F_0(T)=(1-D)F_1(T)$ where $F_1(T)$ are precisely the bounded Lipschitz functions on $\RR$. Now using the fact that $\Lipb(\RR)=\W^{1,\infty}(\RR)$ with equivalent norms we obtain the result. For an alternative proof we refer to \cite[Chapter~II.5(b)]{EN}. Moreover, we obtain
\[
F_{-\alpha}=\Bigl\{f\in\Distr(\RR):\quad F=f-Df\ \text{for}\ f\in\BC^{1-\alpha}(\RR)\Bigl\},
\]
and
\[
X_{-\alpha}=\Bigl\{f\in\Distr(\RR):\quad F=f-Df\ \text{for}\ f\in\h_{b,\text{loc}}^{1-\alpha}(\RR)\Bigl\},
\]
which follow from Corollary \ref{cor:ExtFav}.

We summarize this example by the following diagram:
\[
\BC^1(\RR)\hookrightarrow\Lipb(\RR)\hookrightarrow\h_b^{\alpha}(\RR)\hookrightarrow\h_{b,\text{loc}}^{\alpha}(\RR)\hookrightarrow\BC^{\alpha}(\RR)\hookrightarrow\BUC(\RR)\hookrightarrow\BC(\RR)\hookrightarrow\Ell^{\infty}(\RR)
\]
according to the following abstract chain of spaces
\[
X_1\hookrightarrow F_1\hookrightarrow\xX_{\alpha}\hookrightarrow X_{\alpha}\hookrightarrow F_{\alpha}\hookrightarrow \xX_0\hookrightarrow X_0\hookrightarrow F_0
\]
for $\alpha\in(0,1)$. For the higher order spaces we have:
\begin{align*}
X_n:=\dom(A^n)&=\Bigl\{f\in\BC(\RR):\ f\ \text{is}\ n\text{-times differentiable and}\ f^{(n)}\in\BC(\RR)\Bigr\}\\
&=\Bigl\{f\in\BC(\RR):\ f^{(k)}\in\BC(\RR),\ k=1,\dots,n\Bigr\}=\BC^n(\RR)
\end{align*}
for $n\in\NN$. We denote by $F_{n+\alpha}$ the Favard space which belongs to the restricted semigroup on $X_n$. 
\begin{align*}
F_{n+\alpha}
=\Bigl\{f\in\BC^n(\RR):\ \sup_{\substack{x,y\in\RR\\x\neq y}}{\frac{|f^{(n)}(x)-f^{(n)}(y)|}{|x-y|^{\alpha}}}<\infty\Bigr\}=\BC^{n,\alpha}(\RR)
\end{align*}
This example extends Nagel, Nickel, Romanelli \cite[Sec.{} 3.2]{NagelIdent}.

\subsection{The multiplication semigroup}

Let $\Omega$ the a locally compact space and $X_0=\BC(\Omega)$. Let $q:\Omega\to\CC$ be continuous such that $\sup_{x\in\Omega}{\text{Re}(q(x))}<0$. We define the multiplication operator $M_q:\dom(M_q)\to\BC(\Omega)$ by $M_qf=qf$ on the maximal domain
\[
\dom(M_q)=\{f\in \BC(\Omega):\ qf\in\BC(\Omega)\}.
\] 
This operator generates the semigroup $(T_q(t))_{t\geq0}$ defined by
\[
(T_q(t)f)(x)=\ee^{tq(x)}f(x),\quad t\geq0, x\in\Omega, f\in\BC(\Omega),
\]
which is bi-continuous on $\BC(\Omega)$ with respect to the compact-open topology. Now let $\D=\C(\Omega)$ the space of all continuous functions on $\Omega$, let $\mathcal{M}_q:\C(\Omega)\to\C(\Omega)$ be the multplication operator $\mathcal{M}_qf:=qf$ and $i:\BC(\Omega)\to\C(\Omega)$ the identity. Then by Theorem \ref{thm:iden} we obtain
\[
X_{-1}=\{g\in\C(\mathbb{R}):\ q^{-1}g\in\BC(\RR)\}.
\]
For $\alpha\in(0,1)$, the (abstract) Favard space:
\[
F_{\alpha}=\{f\in\BC(\RR):\ |q|^{\alpha}f\in\BC(\RR)\}.
\]
To see this suppose first that $f\in F_{\alpha}$, hence $\|f\|_{\alpha}<\infty$ which means in particular that
\[
\sup_{t>0}\sup_{x\in\Omega}{\frac{|\ee^{tq(x)}f(x)-f(x)|}{t^{\alpha}}}<\infty.
\]
By specializing $t=\frac{1}{|q(x)|}$ we obtain
\[
{\bigl|\ee^{\frac{q(x)}{|q(x)|}}-1\bigr|\cdot|f(x)|\cdot|q(x)|^{\alpha}},
\]
since
\begin{align}\label{eqn:expf}
\frac{|\ee^{tq(x)}f(x)-f(x)|}{t^{\alpha}}=\frac{|\ee^{tq(x)}-1|\cdot|f(x)||q(x)|^{\alpha}}{|q(x)|^{\alpha}t^{\alpha}}.
\end{align}
Hence $|q|^{\alpha}f\in\BC(\RR)$. For the converse assume that $|q|^{\alpha}f\in\BC(\RR)$. Since the function $g(z)=\frac{|\ee^z-1|}{|z|^{\alpha}}$ is bounded on the left half plane we obtain that $f\in F_{\alpha}$ by \eqref{eqn:expf}. This proves the equality. We also conclude that $F_{\alpha}=X_{\alpha}$ since
\[
\sup_{x\in K}{\left|\frac{\ee^{tq(x)}f(x)-f(x)}{t^{\alpha}}\right|}=\sup_{x\in K}{\left|\frac{\ee^{tq(x)}-1}{tq(x)}\right|\cdot\left|f(x)\right|\cdot\left|q(x)\right|^{\alpha}t^{1-\alpha}}
\]
for each compact set $K\subseteq\Omega$. The extrapolated Favard spaces are then given by
\[
F_{-\alpha}=\bigl\{f\in\BC(\Omega):\ |q|^{1-\alpha} f\in\BC(\Omega)\bigr\}=X_{-\alpha}.
\]
The spaces $\xX_{\alpha}$ are more difficult to describe in general, since the space of strong continuity $\xX_0$ depends substantially on the choice of $q$. For example, if $\frac{1}{q}\in\C_0(\Omega)$, then $\xX_0=\C_0(\RR)$. To see this first notice that $\C_0(\Omega)\subseteq\xX_0$ trivially. On the other hand 
\[
\left|f\right|=\left|\frac{1}{q}\right|\cdot\left|fq\right|
\]
which shows that $\dom(M_q)\subseteq\C_0(\RR)$ and hence that $\xX_0\subseteq\C_0(\RR)$. Now one obtains
\[
\xX_{\alpha}=\{f\in\C_0(\Omega): |q|^{\alpha}f\in\C_0(\Omega)\},
\]
and
\[
\xX_{-\alpha}=\{f\in\C_0(\Omega):\ |q|^{1-\alpha}f\in\C_0(\Omega)\}.
\]
This example extends Nagel, Nickel, Romanelli \cite[Sec.{} 3.2]{NagelIdent}.

\subsection{The Gau\ss{}-Weierstra\ss{} semigroup}
On $X_0=\BC(\RR^d)$ ($d\geq1$) we consider the Gauss-Weierstrass semigroup, defined by $T(0)=\Id$ and
\begin{align}\label{eqn:Gauss}
T(t)f(x)=\frac{1}{(4\pi t)^{\frac{d}{2}}}\int_{\RR^d}{\ee^{-\frac{\left|x-y\right|^2}{4t}}f(y)\ \dd y},\quad t>0,\quad x\in\RR^d,
\end{align}

If one equips $\BC(\RR^d)$ again with the compact-open topology one concludes that $(T(t))_{t\geq0}$ defined by \eqref{eqn:Gauss} is bi-continuous and its space of strong continuity is $\BUC(\RR^d)$. From \cite[Proposition~2.3.6]{LB} we know that the generator $A$ of this semigroup is given $Af=\Delta f$ on the maximal domain
\[
\dom(A)=\{f\in\BC(\RR^d):\quad \Delta f\in\BC(\RR^d)\},
\]
where $\Delta$ is the distributional Laplacian. Now the extrapolation space can again be obtained by Theorem \ref{thm:iden}. Take $\D$ again to be the space of all distributions $\Distr(\RR^d)$ on $\RR^d$ with $\A$ the distributional Laplacian and $i:\BC(\RR^d)\to\Distr(\RR^d)$ the regular embedding. Applying Theorem \ref{thm:iden} we obtain
\[
X_{-1}=\{F\in\Distr(\RR^d):\ F=f-\Delta f\ \text{for some}\ f\in\BC(\RR^d)\}.
\]
The domain of the generator can explicitly be written down, see, e.g., \cite{LB} or\cite{Lunardi}. For $d=1$ the domain is given by
\[
\dom(\Delta)=\BC^2(\RR),
\]
while for $d\geq2$ 
\[
\dom(\Delta)=\Bigl\{f\in\BC(\RR^d)\cap\W^{2,p}_{\text{loc}}(\RR^d),\ \text{for all}\ p\in[1,\infty),\ \text{and}\ \Delta f\in\BC(\RR^d)\Bigl\}.
\]
For $\alpha\in(0,1)\setminus\{\frac{1}{2}\}$ the Favard spaces are 
\[
F_{\alpha}=\BC^{2\alpha}(\RR^d)
\]
while for $\alpha=\frac{1}{2}$ one obtains
\[
F_{\frac{1}{2}}=\Bigl\{f\in\BC(\RR^d):\ \sup_{x\neq y}{\frac{|f(x)+f(y)-2f(\frac{x+y}{2})|}{|x-y|}}<\infty\Bigl\}.
\]
From Corollary \ref{cor:ExtFav} it follows that
\[
F_{-\alpha}=\Bigl\{F\in\Distr(\RR^d):\ F=f-\Delta f\ \text{for some}\ f\in\BC^{2(1-\alpha)}(\RR^d)\Bigr\}
\]
and
\[
F_{-\frac{1}{2}}=\Bigl\{F\in\Distr(\RR^d):\ F=f-\Delta f\ \text{for some}\ f\in F_{\frac{1}{2}}\Bigr\}.
\]

\subsection{The left implemented semigroup} \label{subs:implem}
Let $X_0:=\LLL(E)$ the space of bounded linear operator on a Banach space $E$. We equip $\LLL(E)$ with the operator norm and the strong topology $\tau_{\text{stop}}$ generated by the family of seminorms $\semis=\{p_x:\ x\in E\}$, where
\[
p_x(B)=\|Bx\|,\quad B\in\LLL(E).
\]
Let $(S(t))_{t\geq0}$ be a $C_0$-semigroups with negative growth bound on a Banach space $E$. The semigroup $(\mathcal{U}(t))_{t\geq0}$ on $X_0$ defined by 
\[
\mathcal{U}(t)B=S(t)B,\quad B\in X_0,\quad t\geq0,
\]
is called the semigroup left implemented by $(S(t))_{t\geq0}$. Note that $(\mathcal{U}(t))_{t\geq0}$ has negative growth bound and is a bi-continuous semigroup if $(S(t))_{t\geq0}$ is a $C_0$-semigroup.
We determine the intermediate and extrapolation spaces for this semigroup. We can write:
\begin{align*}
\|B\|_{F_{\alpha}(\mathcal{U})}&=\sup_{t>0}{\frac{\|\mathcal{U}(t)B-B\|}{t^{\alpha}}}=\sup_{t>0}{\frac{\|S(t)B-B\|}{t^{\alpha}}}\\
&=\sup_{t>0}{\sup_{\|x\|\leq1}{\frac{\|S(t)Bx-Bx\|}{t^{\alpha}}}}
=\sup_{\|x\|\leq1}{\sup_{t>0}{\frac{\|S(t)Bx-Bx\|}{t^{\alpha}}}}
=\sup_{\|x\|\leq1}{\|Bx\|_{F_{\alpha}(S)}}.
\end{align*}
From this we conclude the following.
\begin{proposition}\label{prop:implFav}
Let $(\mathcal{U}_L(t))_{t\geq0}$ be the semigroup which is left-implemented by $(S(t))_{t\geq0}$. Then \[F_{\alpha}(\mathcal{U})=\LLL(E,F_{\alpha}(S))\] with the same norms.\end{proposition}

From the definition we obtain:
\begin{align*}
X_{\alpha}(\mathcal{U})&=\Bigl\{B\in\LLL(E):\ \tlim_{t\to0}{\frac{\mathcal{U}_L(t)B-B}{t^{\alpha}}}=0,\ \|B\|_{F_{\alpha}(\mathcal{U})}<\infty\Bigl\}\\
&=\Bigl\{B\in\LLL(E):\ \lim_{t\to0}{\frac{\|S(t)Bx-Bx\|}{t^{\alpha}}}=0\quad \text{for all } x\in E\Bigl\},\\
\xX_{\alpha}(\mathcal{U})&=\Bigl\{B\in \LLL(E):\ \lim_{t\to0}{\frac{S(t)B-B}{t^{\alpha}}}=0\Bigl\}.
\end{align*}

\begin{proposition}Let $(\mathcal{U}(t))_{t\geq0}$ be the semigroup which is left-implemented by $(S(t))_{t\geq0}$. Then \[X_{\alpha}(\mathcal{U})=\LLL(E,X_{\alpha}(S))\] with the same norms.\end{proposition}

After the discussion of abstract Favard and H\"older spaces of the implemented semigroup, we turn to extrapolation spaces. Theses spaces have been studied by Alber in \cite{Alber2001} but only for the $C_0$-semigroup $(\underline{\mathcal{U}}(t))_{t\geq 0}$ on the space $\xX_0=\overline{\dom(\mathcal{G})}$ which depends on the semigroup $(\underline{\mathcal{U}}(t))_{t\geq 0}$. First we recall a result from \cite{Alber2001}: The generator $\mathcal{G}$ of $(\mathcal{U}(t))_{t\geq0}$ is given by
\[
\mathcal{G}V=A_{-1}V,
\]
on
\[
\dom(\mathcal{G})=\left\{V\in\LLL(E):\quad A_{-1}V\in\LLL(E)\right\},
\]
where $A_{-1}$ denotes the generator of the extrapolated $C_0$-semigroup $(S_{-1}(t))_{t\geq 0}$ on $E_{-1}$.
The extrapolation spaces $X_{-1}$ and $\xX_{-1}$ can now we obtained by Theorem \ref{thm:iden}. For that let
\[
\D=\bigl\{S:E\to E_{-\infty}:\ \text{linear and continuous}\bigr\},
\]
where $E_{-\infty}$ is the universal extrapolation space of $(S(t))_{t\geq 0}$ (see the paragraph preceding Theorem \ref{thm:iden}),
and let $i:\LLL(E)\to\D$ be the identity. Consider the operator-valued multiplication operator
\[
\mathcal{A}V=A_{-\infty}V,\quad V\in\D
\]
where $A_{-\infty}x=A_{-(n-1)}x$ for $x\in E_{-n}$. 
Notice that $\lambda-\mathcal{A}:X_0\to\D$ is injective for $\lambda>0$ since $A_{-\infty}$ and $A_{-1}$ coincide on $E$. Hence by applying Theorem \ref{thm:iden} we obtain
\[
X_{-1}=\left\{A_{-1}V:\quad V\in\LLL(E)\right\} 
\]
and
\[
\xX_{-1}=\left\{A_{-1}V:\quad V\in\xX_0\right\}. 
\]
From this we conclude the following description for $X_{-1}$:
\[
X_{-1}=\Bigl\{V\in\LLL(E,E_{-1}):\ \exists(V_n)_{n\in\NN}\subseteq\LLL(E)\text{ with } V_n\to V\text{ strongly}\Bigr\}=\overline{\LLL(E)}^{\LLL_{\mathrm{stop}}(E,E_{-1})}.
\]
And similarly for $\xX_{-1}$:
\begin{align*}
\xX_{-1}&=\Bigl\{V\in\LLL(E,E_{-1}):\ \exists(V_n)_{n\in\NN}\subseteq\LLL(E)\text{ with } V_n\to V\text{ in $\LLL(E,E_{-1})$}\Bigr\}=\overline{\LLL(E)}^{\LLL(E,E_{-1})}.
\end{align*}
This is a result of Alber, see \cite{Alber2001}, which we could recover as a simple consequence of the abstract techniques described in this paper.
Finally, we obtain by Corollary \ref{cor:ExtFav} that for $\alpha\in (0,1)$
\[
F_{-\alpha}(\mathcal{U})=A_{-1}\LLL(E,F_{1-\alpha}(S)) \quad\text{and}\quad X_{-\alpha}(\mathcal{U})=A_{-1}\LLL(E,X_{1-\alpha}(S)).
\]



\begin{thebibliography}{10}\normalsize

\bibitem{adler2014}
M.~Adler, M.~Bombieri, and K.-J. Engel, \emph{On perturbations of generators of
  ${C}_{0}$ -semigroups}, Abstr. Appl. Anal. \textbf{2014} (2014), 13 pages.

\bibitem{Alber2001}
J.~Alber, \emph{On implemented semigroups}, Semigroup Forum \textbf{63} (2001),
  no.~3, 371--386.

\bibitem{Amann}
H.~{Amann}, \emph{{Parabolic evolution equations in interpolation and
  extrapolation spaces}}, {J. Funct. Anal.} \textbf{78} (1988), no.~2, 233--270
  (English).

\bibitem{Arendt94}
W.~Arendt, O.~El-Mennaoui, and V.~K\'eyantuo, \emph{Local integrated
  semigroups: evolution with jumps of regularity}, J. Math. Anal. Appl.
  \textbf{186} (1994), no.~2, 572--595.

\bibitem{Cerrai}
S.~Cerrai, \emph{A {H}ille-{Y}osida theorem for weakly continuous semigroups},
  Semigroup Forum \textbf{49} (1994), no.~3, 349--367.

\bibitem{CooperSaks}
J.~B. Cooper, \emph{Saks spaces and applications to functional analysis},
  second ed., North-Holland Mathematics Studies, vol. 139, North-Holland
  Publishing Co., Amsterdam, 1987, Notas de Matem\'atica [Mathematical Notes],
  116.

\bibitem{DaPrato1982}
G.~Da~Prato and P.~Grisvard, \emph{On extrapolation spaces}, Atti della
  Accademia Nazionale dei Lincei. Classe di Scienze Fisiche, Matematiche e
  Naturali. Rendiconti Lincei. Matematica e Applicazioni \textbf{72} (1982),
  no.~6, 330--332 (eng).

\bibitem{Desch}
W.~{Desch} and W.~{Schappacher}, \emph{{On relatively bounded perturbations of
  linear ${C}\sb 0$-semigroups}}, {Ann. Sc. Norm. Super. Pisa, Cl. Sci., IV.
  Ser.} \textbf{11} (1984), 327--341 (English).

\bibitem{EN}
K.-J. Engel and R.~Nagel, \emph{One-parameter semigroups for linear evolution
  equations}, Graduate Texts in Mathematics, vol. 194, Springer-Verlag, New
  York, 2000, With contributions by S. Brendle, M. Campiti, T. Hahn, G.
  Metafune, G. Nickel, D. Pallara, C. Perazzoli, A. Rhandi, S. Romanelli and R.
  Schnaubelt.

\bibitem{Esterle}
J.~Esterle, \emph{Mittag-{L}effler methods in the theory of {B}anach algebras
  and a new approach to {M}ichael's problem}, Proceedings of the conference on
  {B}anach algebras and several complex variables ({N}ew {H}aven, {C}onn.,
  1983), Contemp. Math., vol.~32, Amer. Math. Soc., Providence, RI, 1984,
  pp.~107--129.

\bibitem{FaPHD}
B.~Farkas, \emph{Perturbations of bi-continuous semigroups}, Ph.D. thesis,
  E\"otv\"os Lor\'and University, 2003.

\bibitem{FaStud}
B.~Farkas, \emph{Perturbations of bi-continuous semigroups}, Studia Math.
  \textbf{161} (2004), no.~2, 147--161.

\bibitem{Greiner}
G.~{Greiner}, \emph{{Perturbing the boundary conditions of a generator}},
  {Houston J. Math.} \textbf{13} (1987), 213--229 (English).

\bibitem{Haase}
M.~Haase, \emph{Operator-valued ${H}^{\infty}$-calculus in inter- and
  extrapolation spaces}, Integral Equations and Operator Theory \textbf{56}
  (2006), no.~2, 197--228.

\bibitem{JWW}
B.~{Jacob}, S.-A. {Wegner}, and J.~{Wintermayr}, \emph{{Desch-Schappacher
  perturbation of one-parameter semigroups on locally convex spaces}}, {Math.
  Nachr.} \textbf{288} (2015), no.~8-9, 925--934 (English).

\bibitem{Kraaij2016}
R.~Kraaij, \emph{Strongly continuous and locally equi-continuous semigroups on
  locally convex spaces}, Semigroup Forum \textbf{92} (2016), no.~1, 158--185.

\bibitem{Ku}
F.~K\"uhnemund, \emph{A {H}ille-{Y}osida theorem for bi-continuous semigroups},
  Semigroup Forum \textbf{67} (2003), no.~2, 205--225.

\bibitem{Kunze2009}
M.~Kunze, \emph{Continuity and equicontinuity of semigroups on norming dual
  pairs}, Semigroup Forum \textbf{79} (2009), no.~3, 540.

\bibitem{LB}
L.~Lorenzi and M.~Bertoldi, \emph{Analytical methods for {M}arkov semigroups},
  Monographs and Research Notes in Mathematics, Taylor \& Francis, 2006.

\bibitem{Lunardi}
A.~Lunardi, \emph{Interpolation theory}, second ed., Appunti. Scuola Normale
  Superiore di Pisa (Nuova Serie). [Lecture Notes. Scuola Normale Superiore di
  Pisa (New Series)], Edizioni della Normale, Pisa, 2009.

\bibitem{NagelSobolev}
R.~{Nagel}, \emph{Sobolev spaces and semigroups}, Semesterbericht
  Funktionalanalysis (1983).

\bibitem{NagelSurvey}
R.~{Nagel}, \emph{{Extrapolation spaces for semigroups}}, RIMS
  K\^{o}ky\^{u}roku \textbf{1009} (1997), 181--191 (English).

\bibitem{NagelIdent}
R.~{Nagel}, G.~{Nickel}, and S.~{Romanelli}, \emph{{Identification of
  extrapolation spaces for unbounded operators}}, {Quaest. Math.} \textbf{19}
  (1996), no.~1-2, 83--100 (English).

\bibitem{nagel1993inhomogeneous}
R.~Nagel and E.~Sinestrari, \emph{Inhomogeneous volterra integrodifferential
  equations for hille-yosida operators}, Lecture Notes in Pure and Applied
  Mathematics (1993), 51--51.

\bibitem{NS}
R.~Nagel and E.~Sinestrari, \emph{Extrapolation spaces and minimal regularity
  for evolution equations}, J. Evol. Equ. \textbf{6} (2006), no.~2, 287--303.

\bibitem{DAPRATO1984107}
G.~D. Prato and P.~Grisvard, \emph{Maximal regularity for evolution equations
  by interpolation and extrapolation}, Journal of Functional Analysis
  \textbf{58} (1984), no.~2, 107 -- 124.

\bibitem{Priola}
E.~Priola, \emph{On a class of {M}arkov type semigroups in spaces of uniformly
  continuous and bounded functions}, Studia Math. \textbf{136} (1999), no.~3,
  271--295.

\bibitem{Sinestrari1996}
E.~Sinestrari, \emph{Interpolation and extrapolation spaces in evolution
  equations}, pp.~235--254, Birkh{\"a}user Boston, Boston, MA, 1996.

\bibitem{Staffans2004}
O.~J. Staffans and G.~Weiss, \emph{Transfer functions of regular linear systems
  part {III}: Inversions and duality}, Integral Equations and Operator Theory
  \textbf{49} (2004), no.~4, 517--558.

\bibitem{van1992adjoint}
J.~van Neerven, \emph{The adjoint of a semigroup of linear operators}, Lecture
  Notes in Mathematics, Springer Berlin Heidelberg, 1992.

\bibitem{Walter}
T.~Walter, \emph{St\"orungstheorie von {G}eneratoren und {F}avardklassen},
  Semesterbericht Funktionalanalysis (1986).

\bibitem{W}
S.-A. Wegner, \emph{Universal extrapolation spaces for
  $\hbox{$C_0$}$-semigroups}, Annali dell 'Universita' di ferrara \textbf{60}
  (2014), no.~2, 447--463.

\end{thebibliography}

\providecommand{\bysame}{\leavevmode\hbox to3em{\hrulefill}\thinspace}
\providecommand{\MR}{\relax\ifhmode\unskip\space\fi MR }
\providecommand{\MRhref}[2]{%
  \href{http://www.ams.org/mathscinet-getitem?mr=#1}{#2}
}
\providecommand{\href}[2]{#2}

\parindent0pt
\end{document}